\documentclass[a4paper, portrait]{amsart}

\usepackage{amsmath,amssymb,graphicx,epsfig}
\usepackage{amsmath}
\usepackage{amssymb}
\usepackage{graphicx}
\usepackage{subfigure}
\usepackage{graphicx}
\usepackage{eucal}
\usepackage{mathrsfs}
\usepackage {epsfig}
\usepackage {graphicx}
\usepackage{epstopdf}
 \newtheorem{thm}{Theorem}[section]
 \newtheorem{cor}[thm]{Corollary}
 \newtheorem{lem}[thm]{Lemma}
 \newtheorem{prop}[thm]{Proposition}
 \newtheorem{claim}[thm]{Claim}
 \newtheorem{defn}[thm]{Definition}
 \newtheorem{rem}[thm]{Remark}

\newcommand{\eps}{\varepsilon}

\newcommand \Lcal   {\mathcal L}

\newcommand{\suml}{\sum\limits}
\newcommand{\fu}{\mathfrak{u}}
  \newcommand{\fv}{\mathfrak{v}}
  
  \newcommand{\fw}{\mathfrak{w}}
  \newcommand{\fz}{\mathfrak{z}}
  \newcommand{\ff}{\mathfrak{f}}
\newcommand \be     {\begin{equation}}
\newcommand \ee     {\end{equation}}

\newcommand {\parx}{\frac{d}{dx}}
\newcommand {\partt}{\frac{d}{dt}}

\newcommand \Acal   {\mathcal A}

 \newcommand{\set}[1]{\left\{#1\right\}}

 \numberwithin{equation}{section}
\begin{document}

\title [SPLINES, DISCRETE BIHARMONIC and EIGENVALUES]
 { SPLINE FUNCTIONS, THE DISCRETE BIHARMONIC OPERATOR AND APPROXIMATE EIGENVALUES}

\author{Matania Ben-Artzi}
\address{Matania Ben-Artzi: Institute of Mathematics, The Hebrew University, Jerusalem 91904, Israel}
\email{mbartzi@math.huji.ac.il}
\author{Guy Katriel}
\address{Guy Katriel: Department of Mathematics, Ort Braude College, Karmiel 21982, Israel}
\email{katriel@braude.ac.il}


\thanks{It is a pleasure to thank Jean-Pierre Croisille, Dalia Fishelov and Robert Krasny for very fruitful discussions. We thank M. Hansmann for calling our attention to the paper of Markus~\cite{markus}. }



\date{\today}




\begin{abstract}
The biharmonic operator plays a central role in a wide array of physical models, notably in elasticity theory and the streamfunction formulation of the Navier-Stokes equations. The need for corresponding numerical simulations has led, in recent years, to the development of a discrete biharmonic calculus. The primary object of this calculus is a high-order compact discrete biharmonic operator  (DBO). The numerical results have been remarkably accurate, and have been corroborated by some rigorous proofs. However, there remained the ``mystery`` of the ``underlying reason'' for this success. This paper is a contribution in this direction,  expounding the strong connection between cubic spline functions (on an interval) and the  DBO. It is shown in particular that the (scaled) fourth-order distributional derivative of the cubic spline is identical to the action of the DBO on  grid functions. The DBO is constructed in terms of the discrete Hermitian derivative.  A remarkable fact is that the kernel of the inverse of the discrete operator is (up to scaling) equal to the  grid evaluation of the  kernel of $\Big[\Big(\parx\Big)^4\Big]^{-1} .$  Explicit expressions are presented for both  kernels. The relation between the (infinite) set of eigenvalues of the fourth-order Sturm-Liouville problem and the finite set of eigenvalues of the discrete biharmonic operator is studied, and the discrete eigenvalues are proved to converge (at an ``optimal'' $O(h^4)$ rate) to the continuous ones. Another remarkable consequence is the validity of a \textit{comparison principle.}
It is well known that there is no maximum principle for the fourth-order equation.  However, a positivity result is derived, both  for the
continuous and the discrete biharmonic equation, showing that in both cases the kernels are order preserving.

\end{abstract}

\maketitle
\section{\textbf{INTRODUCTION}}
The operator $\Big(\parx\Big)^4$ on the interval $[0,1]$ is certainly the simplest conceivable example of a fourth-order elliptic one-dimensional operator. As such, its spectral theory is very well understood  ~\cite[Chapter 5]{davies} or ~\cite{everitt}. In classical terminology, its study is labeled as a ``fourth-order Sturm-Liouville theory''. The numerical computation of the eigenvalues was carried out using a ``Shannon-type'' sampling method in ~\cite{boumenir} , by ``matrix methods'' in ~\cite{rattana} and by finite element methods in ~\cite{andrew}.  On the other hand, the analogous  ``discrete'' treatment leaves much to be desired. By this we mean the construction of approximating finite difference operators so that their eigenvalues can be shown to converge to the spectrum of the differential operator. Clearly, a significant question
is the possibility of obtaining a high rate of convergence in this case.

The goal of this paper is to  fill this gap, by exploring the very interesting structural similarities between $\Big(\parx\Big)^4$ and a suitable discrete biharmonic operator (DBO)  $\delta_x^4.$ The bridge between ``continuous'' and `` discrete'' is achieved by using the classical cubic spline functions. It is well known that the convergence of finite-dimensional approximations to an infinite-dimensional, unbounded differential operator, does not entail the convergence of the respective spectra. One of the main results here is to provide an affirmative answer to this convergence property. The proof relies on the strong connection between the DBO and differential operations on spline functions.

 A basic tool is the \textit{ discrete Hermitian derivative} on an interval, that gives a fourth-order accurate approximation to the derivative of a smooth function. It has been the cornerstone in the construction of a fourth-order discrete approximation to the one-dimensional biharmonic operator ~\cite{abarbanel} and its extension to the full fourth-order Sturm-Liouville problem ~\cite{sturm}.  In the two-dimensional case, it has been used in the construction of a compact high-order finite difference scheme for the Navier-Stokes system in the pure streamfunction formulation ~\cite[Part II]{book}. In this paper the full discrete elliptic theory of the DBO is exploited in the study of the discrete spectrum and its asymptotic behavior as the number of grid points increases to infinity.

    The structure of the paper is as follows.

    In Section ~\ref{secsetupcubic} we recall the basic (classical) construction of cubic spline functions on an interval. For the convenience of the reader we provide the full (and standard) proofs of the essential properties of these functions, those that are used in the sequel.

    In Section ~\ref{subsecdiffoperators} we recall the definitions of the discrete finite difference operators, and in particular introduce the Hermitian derivative and the discrete biharmonic operator $\delta_x^4.$

    In Section ~\ref{secsplinehermit} we establish the equality of the Hermitian derivative and the derivative of the interpolating cubic spline. This is a fundamental fact connecting the two non-local fourth-order approximations of the derivative. We were unable to locate this remarkable fact in the literature, even though we are still convinced that such a classical fact should be well-known.

    Then the connection between the discrete biharmonic operator $\delta_x^4$  and the interpolating cubic spline function is established. It is in fact the main theme of this paper. Recall that the cubic spline is a $C^2$ functions, with finite jumps of the third-order derivatives at grid points. The result here (Proposition ~\ref{cor3rdspljump}) is that the sizes of these jumps are determined by the DBO acting on the grid values. We have not been able to locate such a result in the literature, even though it seems to be such a fundamental fact.

    This connection enables us to prove, in Section ~\ref{secpositivity} , positivity results for the continuous and discrete fourth-order operators (see Proposition ~\ref{propmaximumcont} and Proposition ~\ref{propmaximumdisc}). Recall that there is no maximum principle for the fourth-order operator. Once again, it seems to us that the positivity result should exist already in the literature, but we have not been able to locate it.

    In Section ~\ref{seckernel} we first give the explicit form of the kernel (Green's function) of the continuous operator. In the first instance, this kernel acts in $L^2(0,1).$ We then extend it to  the negative Sobolev space $H^{-2}(0,1).$ This space includes all finite measures, and in particular all grid functions. Using the connection to cubic spline functions we establish the remarkable result that the discrete resolvent (namely, the kernel of  $(\delta_x^4)^{-1}$) is just the grid evaluation of the continuous kernel, up to scaling. Indeed, this can be viewed as an alternative, very natural, definition of the compact discrete biharmonic operator.

    Finally, Section ~\ref{seceigenvalues} is concerned with the \textit{eigenvalues} of both the continuous and discrete operators. These eigenvalues (more precisely their inverses) are studied in terms of the ``kernel tools'' developed in the previous sections; the established connection between the discrete and continuous kernels implies that the discrete eigenvalues are actually obtained by a ``Nystr\"{o}m method''~\cite{nystrom}.

    The highlight of this section (and one of the main results of the entire paper) is the proof of the convergence of the discrete eigenvalues  to the continuous ones, at an ``optimal'' fourth-order rate (Theorem ~\ref{thmoptimalallevs}). This result is  obtained by combining two ingredients:
     \begin{itemize}
     \item A suitable adaptation (Lemma ~\ref{lemconvergence}) of  a more general abstract convergence theorem ~\cite{kato, markus}. However, we have chosen to provide a self-contained, much simpler, proof, that builds on the analytic theory of finite-dimensional perturbations, as expounded in Kato's classical book ~\cite{kato-book}.
         \item The dependence of the eigenvalues on the respective kernels, see Proposition ~\ref{propdistLambdah}.
     \end{itemize}

In Appendix ~\ref{sechaggaiexplicit} we use the approach of ``generating polynomials'' in order to give yet another explicit construction of the kernel of the discrete resolvent $(\delta_x^4)^{-1}.$ In fact, this classical method enables us to establish a totally different point-of-view concerning the compact discrete operators used here, beginning with the Hermitian derivative. This approach has the advantage of being directly related to the definitions of the discrete operators, avoiding the ``mediation'' of spline functions. It
    is potentially applicable as a computational approach to similar (discrete) problems.

\section{\textbf{THE BASIC SETUP for CUBIC SPLINES}}\label{secsetupcubic}
In what follows we consider the interval $\Omega=[0,1]$ with a uniform grid
  $$x_j=jh,\quad 0\leq j\leq N,\quad h=\frac{1}{N}.$$
  We fix values $f=\set{f_j}^N_{j=0}$ so that $f_0=f_N=0,$ and consider the family
  $$\Acal=\set{u\in H^2_0(\Omega),\quad u_j=f_j,\quad j=0,1,...,N}.$$

  The space $H^2_0(\Omega)$ is the space of functions having first and second (distrbutional) derivatives in $L^2(\Omega)$
  and vanishing, with their first-order derivatives, at the endpoints.

  It is well known that the norm in $H^2_0(\Omega)$ can be defined by
  $$\|u\|_{H^2_0(\Omega)}^2=\int_0^1|u''(x)|^2dx,$$
  and we shall refer henceforth to this norm.

  We consider the functional
  $$I(u)=\int\limits_0^1|u''(x)|^2dx,\quad u\in H^2_0(\Omega).$$
       We are interested in a minimizer for this functional, restricted to $\Acal.$ Since the properties of this minimizer will be essential in the rest of this paper, we provide here the details of the proof of this classical fact of the calculus of variations. A purely algebraic proof can be found in ~\cite[Theorem 3.4.3]{spline-book} or ~\cite[Chapter IV, Cubic Spline Interpolation]{deboor0}.
   \begin{claim} The functional has a unique minimizer on $\Acal$ , which we designate as $s_f,$
      $$I(s_f)<I(g),\quad s_f\ne g\in\Acal.$$
      \end{claim}
      \begin{proof}
        In fact, the functional $I(u)$ is strictly convex and $\Acal$ is convex, so the existence of a unique minimizer is guaranteed by general principles.

        However, for the convenience of the reader, we provide a simple, straightforward proof.
      Since clearly $\Acal\neq \emptyset,$ and $I(u)\geq 0,$ we can define
        $$I_{min}=\inf\limits_{u\in\Acal} I(u).$$
         Let $\set {u_n}_{n=1}^\infty\subseteq\Acal$ be a sequence such that
          $$I(u_n)\xrightarrow[n\to\infty]{}I_{min}.$$

             The boundedness of  $\set {u_n}_{n=1}^\infty$ in the Hilbert space $H^2_0(\Omega)$ implies
            ~\cite[Appendix D.4]{evans} that, again passing to a subsequence without changing index, there is a weak limit,
            $$u_n\xrightarrow[n\to\infty]{w}v\in H^2_0(\Omega).$$
            This weak limit satisfies
            $$I(v)\leq \liminf\limits_{n\to\infty}I(u_n),$$
            so that it is indeed a minimizer.

%

            To prove the uniqueness of such a minimizer, suppose that $w\in H^2_0(\Omega)$ is another minimizer.
             Let $r=\frac{v+w}{2}.$  Clearly $r\in\Acal,$ and,
                       by the Cauchy-Schwarz inequality,
             $$\int_0^1|r''(x)|^2dx=\frac14\Big[\int_0^1|v''(x)|^2dx+\int_0^1|w''(x)|^2dx+2\int_0^1v''(x)w''(x)dx\Big]\leq I_{min}.$$
             It follows that $r$ is also a minimizer and in particular
             $$\int_0^1v''(x)w''(x)dx=I_{min}=\Big[\int_0^1|v''(x)|^2dx\Big]^{\frac12}\Big[\int_0^1|w''(x)|^2dx\Big]^{\frac12}.$$
             As is well known, equality in the Cauchy-Schwarz inequality implies that $w''=\pm v''.$ The boundary conditions now yield $w=\pm v,$ and the constraints at the nodes finally force $w=v.$
      \end{proof}
      \begin{claim}\label{sf}
      \begin{enumerate}
      \item $s_f$ is a cubic polynomial in each interval $[x_j,x_{j+1}],\quad j=0,1,...,N-1.$
      \item $s_f\in C^2_0(\Omega).$
      \item The previous two properties , supplemented by the constraints $s_f(x_j)=f_j,\quad j=1,...,N-1, $ and $s_f(x_0)=s'_f(x_0)=s_f(x_N)=s'_f(x_N)=0 $ determine $s_f$ uniquely.
      \end{enumerate}
      \end{claim}
      \begin{proof}
      The basic property of $s_f$ is that
       $$\int\limits_0^1s_f''(x)v''(x) dx=0,$$
       for all $v\in H^2_0(\Omega)$ vanishing at all nodes $x_j,\quad 0\leq j\leq N.$

       We obtain the first property by taking any $v\in C^\infty_0(x_j,x_{j+1}),$ and integrating twice by parts. The second property follows by observing that the test functions can have arbitrary values $v'(x_j),v''(x_j).$

       Finally, consider the space  $V\subseteq\set{s\in C^2(\Omega)\cap H^2_0(\Omega)},$  such that $s\in V$ is a cubic polynomial in each interval $[x_j,x_{j+1}],\quad j=0,1,...,N-1.$ Since $s\in V$ has four parameters in each interval $[x_j,x_{j+1}],$ subtracting the number of constraints at all interior nodes and the endpoints yields
         $$dim \,V\leq 4N-3(N-1)-4=N-1.$$
          Since on the other hand we have for every $f=\set{f_j}^N_{j=0}$ so that $f_0=f_N=0,$ a corresponding $s_f\in V,$ it follows that $dim\,V=N-1,$ so that $s_f$ is the unique function in $V$ satisfying the constraints.
      \end{proof}
      \begin{defn} The function $s_f$ is called the (``type I'') \textbf{cubic spline} corresponding to the constraints
         $$s_f(x_j)=f_j,\quad j=1,...,N-1, \,\,s_f(x_0)=s'_f(x_0)=s_f(x_N)=s'_f(x_N)=0. $$
      \end{defn}
      \begin{claim}\label{linear} Consider the vectors $f=\set{f_j}^N_{j=0}$ such that $f_0=f_N=0.$ Then the map $f\hookrightarrow s_f\in H^2_0(\Omega)$ is one-to-one and linear.
      \begin{rem} When we introduce discrete spaces of grid functions in the following section, we shall denote by $l^2_{h,0}$ the space of such vectors $f.$
      \end{rem}
      \begin{proof} The fact that the map is one-to-one is obvious since $s_f$ determines $f.$ The linearity follows from the uniqueness part in Claim ~\ref{sf}. Indeed, if $s_f,s_g$ correspond to $f,g=\set{f_j,g_j}^N_{j=0}$ so that $g_0=g_N=f_0=f_N=0,$ respectively, then $s_f+s_g\in V$ (the space introduced in the proof of Claim ~\ref{sf}) and it satisfies the constraints corresponding to $f+g,$ hence $s_{f+g}=s_f+s_g.$

      \end{proof}
      \end{claim}
       \section{\textbf{ SETUP and DEFINITION OF THE DISCRETE OPERATORS}}\label{subsecdiffoperators}
We equip the interval $\Omega=[0,1]$ with a uniform grid
  $$x_j=jh,\quad 0\leq j\leq N,\quad h=\frac{1}{N}.$$
     The approximation is carried out by grid functions $\fv$ defined on $\set{x_j,\,0\leq j\leq
     N}.$ The space of these grid functions is denoted by $l^2_h.$
     For their components we use either $\fv_j$ or $\fv(x_j).$

     For every smooth function $f(x)$ we define its associated grid function
     \be\label{eqdefustar}f^\ast_j=f(x_j),\quad 0\leq j\leq N.\ee

     The discrete $l^2_h$ scalar product is defined by
$$(\fv,\fw)_h=h\suml_{j=0}^N\fv_j\fw_j,$$

     and the corresponding norm  is
     \be\label{eqdefl2norm}|\fv|_h^2=h\suml_{j=0}^N\fv_j^2.\ee

     For linear operators $\mathcal{A}:l^2_h\to l^2_h$ we use
     $|\mathcal{A}|_h$ to denote the operator norm.

     The discrete sup-norm is
      \be\label{eqdeflinfnorm}|\fv|_\infty=\max_{0\leq j\leq N}\set{|\fv_j|}.\ee

      The discrete homogeneous space of grid functions is defined by
      \be\label{eqdefdiscl20}l^2_{h,0}=\set{\fv,\,\,\fv_0=\fv_N=0}.\ee

      Given $\fv\in l^2_{h,0}$ we introduce the basic (central) finite difference operators

       \be\label{eqdefinedelta2}\aligned
       (\delta_x\fv)_j=\frac{1}{2h}(\fv_{j+1}-\fv_{j-1}),\quad 1\leq j\leq N-1,\\
       (\delta^2_x\fv)_j=\frac{1}{h^2}(\fv_{j+1}-2\fv_j+\fv_{j-1}),\quad 1\leq j\leq N-1,\\
        \endaligned\ee

    The cornerstone of our approach to finite difference operators is the introduction of the \textbf{Hermitian derivative} of $\fv\in l^2_{h,0},$ that will replace $\delta_x.$ It will serve not only in approximating (to fourth-order of accuracy) first-order derivatives, but also as a fundamental building block in the construction of finite difference approximations to higher-order derivatives.

    First, we introduce the ``Simpson operator''

  \be\label{eqdefsigmasimpson}(\sigma_x\fv)_j=\frac16\fv_{j-1}+\frac23\fv_{j}+\frac16\fv_{j+1},\quad 1\leq j\leq N-1.\ee
       Note the operator relation (valid in $l^2_{h,0}$)
       \be\sigma_x=I+\frac{h^2}{6}\delta_x^2,
       \ee
       so that $\sigma_x$ is an ``approximation to the identity''.

    The Hermitian derivative $\fv_x$ is now defined by
    \be\label{eqdefhermit} (\sigma_x\fv_x)_j=(\delta_x\fv)_j,\quad 1\leq j\leq N-1.
    \ee
        \begin{rem}\label{remzerodiscderiv} In the definition ~\eqref{eqdefhermit}, the values of $(\fv_x)_j,\,\,j=0,N,$ need to be provided , in order to make
        sense of the left-hand side (for $j=1,N-1$). If not otherwise specified, \textit{we shall henceforth assume that
          $\fv_x\in l^2_{h,0},$ namely}
          $$(\fv_x)_0=(\fv_x)_{N}=0.$$
          \textit{In particular, the linear correspondence $ l^2_{h,0}\ni \fv\to \fv_x\in l^2_{h,0}$ is well defined, but not onto, since $\delta_x$ has a non-trivial kernel.}
          \end{rem}

          The biharmonic discrete operator is given by (for $\fv,\,\fv_x\in l^2_{h,0}),$
    \be \label{d4} \delta_x^4\fv=\frac{12}{h^2}[\delta_x\fv_x-\delta_x^2\fv].
    \ee

    We next introduce a fourth-order replacement to the operator $\delta^2_x$ (see  ~\cite[Equation (10.50)(c)]{book}),
    \be\label{eqdefdelta2tild} (\widetilde{\delta}_x^2\fv)_j=2(\delta_x^2\fv)_j-(\delta_x\fv_x)_j,\quad 1\leq j\leq N-1.\ee

      Note that, in accordance with Remark ~\ref{remzerodiscderiv} the operator $\widetilde{\delta}_x^2$ is defined on grid functions $\fv\in l^2_{h,0},$ so that also $\fv_x\in l^2_{h,0}.$

    The connection between the two difference operators for the second-order derivative is given by
     \be\label{twosecondorder}-\widetilde{\delta_x}^2=-\delta_x^2 +\frac{h^2}{12} \delta_x^4.\ee

   \begin{rem}\label{remhdependence}
     Clearly the operators  $\delta_x,\, \delta_x^2,\, \delta_x^4$
     depend on $h,$ but for notational simplicity this dependence
     is not explicitly indicated.
     \end{rem}

     The fact that the biharmonic discrete operator $\delta_x^4$ is positive (in particular symmetric) is proved in ~\cite[Lemmas 10.9,\,10.10]{book}. Therefore its inverse $\Big(\delta^4_x\Big)^{-1}$ is also positive. In fact, it satisfies a strong coercivity property, that is also established in the aforementioned reference.

    An interpretation  to the finite-difference operators $~\widetilde{\delta_x}^2$ and $\delta_x^4$ is provided by the ``polynomial approach''~\cite[Section 10.3]{book}, as follows.

     Let $q(x)$ be a fourth-order polynomial such that
     $$q(x_{j})=\fv_j,\,\,q(x_{j\pm 1})=\fv_{j\pm 1},\,\,q'(x_{j\pm 1})=(\fv_x)_{j\pm 1}.$$
     Then
     \be  (\widetilde{\delta_x}^2\fv)_j=q''(x_j),\,\,(\delta_x^4\fv)_j=q^{(4)}(x_j).
     \ee
      The discrete biharmonic operator gives a very accurate approximation to the continuous one (``optimal 4-th order accuracy'') , as seen in the following claim ~\cite[Theorem 10.19]{book} .
          \begin{claim}\label{claim21}
      Let $f(x)\in C^4(\Omega),\,\,\Omega=[0,1].$ Let $u(x)$ satisfy
        \be\label{eqfourthorder}\Big(\frac{d}{dx}\Big)^4u(x)=f(x),\ee
        subject to homogeneous boundary conditions
        \be\label{eqbdrydata}u(0)=\parx u(0)=u(1)=\parx u(1)=0.\ee
          Then
        \be\label{eqoptimalerror}
           |u^\ast-(\delta^4_x)^{-1}f^\ast|_\infty=O(h^4).
        \ee
      \end{claim}
      \begin{rem} The ``$O(h^4)$'' here means that there exists a constant $C>0,$ depending only on $f,$ such that for all integers $N>1,$
       $$|u^\ast-(\delta^4_x)^{-1}f^\ast|_\infty\leq Ch^4,\quad h=\frac{1}{N}.$$
       Observe that the grid functions in this estimate are defined on the grid of (the variable) mesh size $h.$
      \end{rem}

\section{\textbf{SPLINES , HERMITIAN DERIVATIVES and the DISCRETE BIHARMONIC OPERATOR}}\label{secsplinehermit}

We use the notation of the previous section.

Let $\fu\in l^2_{h,0}$ be a grid function vanishing at the endpoints and let $s_\fu\in H^2_0(\Omega)$ be the corresponding spline function.

   We use interchangeably the notation $\fu_j=\fu(x_j).$

    Let $\fu_x$ be the Hermitian derivative of $\fu,$ and we set at the endpoints
    \be\label{eqbdrycond}\fu_x(x_0)=s_\fu'(x_0)=0, \fu_x(x_N)=s_\fu'(x_N)=0.\ee
\begin{prop}\label{fuxsprime}
   For all interior nodes, $s_\fu'(x_j)=\fu_x(x_j),\quad 1\leq j\leq N-1.$
\end{prop}
  \begin{proof}
  To simplify notation we shift $x_j=0,$ so we need to show
     \be
     \frac13 s_\fu'(-h)+\frac43 s_\fu'(0)+ \frac13 s_\fu'(h)=\frac{\fu(h)-\fu(-h)}{h}.
     \ee
      The quadratic part of $s_\fu$ is continuous, so the equality for this part follows from Simpson's rule.

      Thus we need only check for $s_\fu(x)=a^\pm x^3 $ for $\pm x>0.$ But this can be verified directly.
  \end{proof}



    In addition to $\fu\in l^2_{h,0},$ let $\fv\in l^2_{h,0}$ be a grid function vanishing at the endpoints and let $s_\fv$ be the corresponding spline function. At the endpoints we impose again the boundary conditions ~\eqref{eqbdrycond}.
    \begin{claim} The map  $(\fu,\fv)\to\int\limits_0^1s_\fu''(x)s_\fv''(x) dx$ is a scalar product on $l^2_{h,0}.$
    \end{claim}
    \begin{proof}
    In view of Claim ~\ref{linear} the map is bilinear. Furthermore , if $\int\limits_0^1|s_\fu''(x)|^2 dx=0,$ then $s_\fu''\equiv 0$ and since $s_\fu\in H^2_0$ it follows that also $s_\fu\equiv 0,$ which implies $\fu=0.$
    \end{proof}

       We denote by $\delta^4_x\fu$ the Stephenson fourth-order derivative of $\fu.$ It is interesting that the scalar product of the previous claim can be expressed in terms of this fourth-order derivative.
       \begin{prop} Let $\fu,\fu_x,\fv,\fv_x\in l^2_{h,0}.$ \newline
                 The discrete scalar product of $\delta^4_x\fu$ and $\fv$ satisfies
          \be\label{delta4}
          (\delta^4_x\fu,\fv)_h=\int\limits_0^1s_\fu''(x)s_\fv''(x) dx.
          \ee
       \end{prop}
       \begin{proof}
           Pick $j\in\set{1,2,...,N-1}$ and let $Q_j(x)$ be the fourth-order polynomial used in the construction of
            $(\delta^4_x\fu)_j,$ namely,
            $$\aligned Q_j(x_j)=\fu_j=s_\fu(x_j),\quad Q_j(x_{j\pm 1})=\fu_{j\pm 1}=s_\fu(x_{j\pm 1}),\\
            Q'_j(x_j)=s'_\fu(x_j),\quad Q'_j(x_{j\pm 1})=s'_\fu(x_{j\pm 1}).\endaligned$$
            Observe that the second line above follows from Proposition \ref{fuxsprime}.

            Consider the polynomial $Q_j-s_\fu$ in the interval $[x_j,x_{j+1}].$ It is a fourth-order polynomial with double zeros at $x_j, x_{j+1},$ so it must have the form
              \be\label{Qs1}Q_j(x)-s_\fu(x)=A_j(x-x_j)^2(x-x_{j+1})^2,\quad x\in [x_j,x_{j+1}],\ee
              and similarly
               \be\label{Qs2}Q_j(x)-s_\fu(x)=A_{j-1}(x-x_j)^2(x-x_{j-1})^2,\quad x\in [x_{j-1},x_j].\ee
                However,
                \be\label{eqAjeqsteph} A_{j-1}=A_j=\frac{1}{24}Q_j^{(4)}(x_j)=\frac{1}{24}\delta^4_x\fu_j,\ee
                by definition of the discrete biharmonic operator.

                Let us now compute
                $$\aligned\int\limits_{x_j}^{x_{j+1}}s_\fu''(x)s_\fv''(x) dx=s_\fu''(x_{j+1})s_\fv'(x_{j+1})-
                s_\fu''(x_{j})s_\fv'(x_{j})-\int\limits_{x_j}^{x_{j+1}}s_\fu'''(x)s_\fv'(x) dx\\=
                s_\fu''(x_{j+1})s_\fv'(x_{j+1})-
                s_\fu''(x_{j})s_\fv'(x_{j})-s_\fu'''(x_{j+1}^-)s_\fv(x_{j+1})+
                s_\fu'''(x_{j}^+)s_\fv(x_{j}),\endaligned$$
                since the fourth-order derivative of $s_\fu$ vanishes identically in the interval.

                By summation, and recalling that $s_\fu\in C^2,$ we get
                \be\label{fufv}
                \int\limits_0^1s_\fu''(x)s_\fv''(x) dx=\sum\limits_{j=0}^{N-1}(s_\fu'''(x_{j}^+)-s_\fu'''(x_{j}^-))s_\fv(x_{j}).
                \ee
                  From Equations ~\eqref{Qs1},~\eqref{Qs2} we get
                  \be\label{eq3rdderiv}\aligned Q_j'''(x_j)-s_\fu'''(x_j^+)=-12hA_j,\\
                   Q_j'''(x_j)-s_\fu'''(x_j^-)=12hA_j,\endaligned\ee
                   and inserting this in Equation ~\eqref{fufv} yields
                   \be\int\limits_0^1s_\fu''(x)s_\fv''(x) dx=24h\sum\limits_{j=0}^{N-1}A_js_\fv(x_{j})=
                   h\sum\limits_{j=0}^{N-1}(\delta^4_x\fu)_j\fv_j.\ee

       \end{proof}
           \begin{prop}\label{cor3rdspljump}  The jump of the third order derivatives of the cubic splines at the nodes is given by
           \be\label{eqjumpsteph}s_\fu'''(x_j^+)-s_\fu'''(x_j^-)=h(\delta^4_x\fu)_j.\ee
           \end{prop}
           \begin{proof}
              Combine Equations ~\eqref{eq3rdderiv} and ~\eqref{eqAjeqsteph}.
           \end{proof}
           \begin{rem} In the  literature (e.g. ~\cite{spline-book, deboor0} one can find various expressions for the jump of the third order derivatives of the cubic spline. However Proposition ~\ref{cor3rdspljump} provides a new expression, that can be interpreted as a ``fourth-order derivative'' of the function at the node.
           \end{rem}
               We can also interpret the second derivative of $s_\fu$ in terms of the finite difference operators. Recall that this derivative is continuous at the nodes.
               \begin{cor}
                  The value of $s_\fu''(x_j)$ is given by
                  \be\label{eqsfu2ndderiv}
                  s_\fu''(x_j)=(\widetilde{\delta_x}^2\fu)_j-\frac{h^2}{12}(\delta^4_x\fu)_j.
                  \ee
               \end{cor}
               \begin{proof}
               From Equation ~\eqref{Qs1} we get
               $$Q_j''(x_j)-s_\fu''(x_j)=2A_jh^2.$$
               By definition, $Q_j''(x_j)=(\widetilde{\delta_x}^2\fu)_j$ and from
               ~\eqref{eqAjeqsteph} we have $ A_j=\frac{1}{24}Q_j^{(4)}(x_j)=\frac{1}{24}(\delta^4_x\fu)_j,$ hence
               $$s_\fu''(x_j)=Q_j''(x_j)-2A_jh^2=(\widetilde{\delta_x}^2\fu)_j-\frac{h^2}{12}(\delta^4_x\fu)_j.$$
               \end{proof}
               \begin{rem}
               Note that invoking the relation
               ~\eqref{twosecondorder}
                              we obtain from ~\eqref{eqsfu2ndderiv}
              $$ s_\fu''(x_j)=(\delta_x^2\fu)_j-\frac{h^2}{6}(\delta^4_x\fu)_j.$$

               \end{rem}

                \subsection{\textbf{EVALUATING the INTEGRAL $\int\limits_{0}^{1}|s_\fu''(x)|^2dx$}}

       We first compute over a grid interval
                   $$ B_j=\int\limits_{x_j}^{x_{j+1}}|s_\fu''(x)|^2dx,\quad j=0,1,...,N-1. $$
                   To simplify notation, we set $y=x-x_j,$ so that $s(y)=s_\fu(x)$ is a cubic polynomial in $y\in [0,h].$
                   Writing
                   $$s(y)=ay^3+by^2+cy+d,$$
                   we get readily
                   $$s(h)-s(0)=ah^3+bh^2+ch,$$
                   and
                   $$s'(h)+s'(0)=3ah^2+2bh+2c,$$
                   hence
                   \be
                   a=\frac{1}{h^3}[h(s'(h)+s'(0))-2(s(h)-s(0))].
                   \ee
                     Since $s'(y)$ is a quadratic polynomial, we have
                     $$r:=s''(h/2)=\frac{1}{h}(s'(h)-s'(0)),$$
                     and
                     $$s''(y)=r+6a(y-\frac{h}{2}),\quad y\in [0,h].$$
                     Turning now back to the variable $x,$ and taking into account the equalities
                     $$s_\fu(x_j)=\fu(x_j),\quad 0\leq j\leq N-1,$$
                     $$s_\fu'(x_j)=\fu_x(x_j),\quad 0\leq j\leq N-1.$$
                      we obtain

                    \be \label{Bj}\aligned
                     B_j=\int\limits_{x_j}^{x_{j+1}}|s_\fu''(x)|^2dx\hspace{300pt}\\
                     =\frac{1}{h}(\fu_x(x_{j+1})-\fu_x(x_j))^2+\frac{3}{h}[(\fu_x(x_{j+1})+\fu_x(x_j))-
                     2\frac{\fu(x_{j+1})-\fu(x_j)}{h}]^2,\quad j=0,1,...,N-1,
                    \endaligned\ee
                    and
                    \be\label{intfu2prime}
                    \int\limits_0^1|s_\fu''(x)|^2 dx=\sum\limits_{j=0}^{N-1}B_j.
                    \ee
                    \begin{rem}
                      Equation ~\eqref{delta4} can then be used to \textit{define} the discrete fourth-order derivative $\delta^4_x\fu$ when $\fu,\fu_x\in l^2_{h,0}.$ From equation ~\eqref{intfu2prime} we obtain an explicit expression for $\delta^4_x\fu_j,$ which is actually the Stephenson expression.
                    \end{rem}
                    \section{\textbf{POSITIVITY}}\label{secpositivity}
                       It is well known that there is (in general) no maximum principle for elliptic partial differential operators of order higher than two. For the biharmonic equation in multi-dimensional domains there exist versions of the principle that involve estimates of the gradient of the solution, see ~\cite{pipher} and references therein. Under Dirichlet boundary conditions (the only ones considered here) the  \textit{preservation of positivity property} means that $\Delta^2u\geq 0\Rightarrow u\geq 0.$ It is actually a \textit{property of the domain.}  The maximum principle implies preservation of positivity but of course not vice versa. In the multi-dimensional case (excluding the one-dimensional case) we refer to ~\cite{Grunau} and references therein.

                       In our one-dimensional case we have the following proposition.
                       \begin{prop}\label{propmaximumcont} Let
                       $$\Big(\parx\Big)^4u(x)=f(x),$$
                       where $u\in H^4(\Omega)\cap H^2_0(\Omega).$ Then the following comparison principle holds.

                       If $f(x)\geq 0,\,\,x\in\Omega,$ then also $u(x)\geq 0,\,\,x\in \Omega.$
                       \end{prop}
                       \begin{proof}
                         Suppose to the contrary that for some $y\in (0,1)$ we have $u(y)<0.$ We can assume that $y$ is a minimum point for $u,$ so that
                         $$u'(y)=0,\,\,u''(y)\geq 0.$$
                         Since $u'$ vanishes at the endpoints, we infer that there are points
                         $$\xi\in (0,y),\,\,\eta\in (y,1),$$
                         such that
                         $$u''(\xi)=u''(\eta)=0.$$
                         Let
                         \be\label{eqdefab}\aligned
                         a=\inf\Big\{\xi\in\Omega,\,\,u''(\xi)=0\Big\},\\
                         b=\sup \Big\{\eta\in\Omega,\,\,u''(\eta)=0\Big\}.
                         \endaligned\ee
                         Consider the function $v(x)=u''(x).$ It satisfies in the interval $[a,b]$ the inequality
                         $$v''(x)=f(x)\geq 0,$$
                         as well as $v(a)=v(b)=0$ and $v(y)\geq 0.$

                         The standard maximum principle now yields
                         $$v(x)\equiv 0,\,\,x\in [a,b],$$
                         hence also $u'(x)\equiv u'(y)= 0,\,x\in[a,b].$

                         If $a>0$ we get a contradiction since there is a point $\xi\in (0,a)$ with $u''(\xi)=0.$ Similarly if $b<1.$ We conclude that $u'(x)\equiv 0,\,x\in[0,1],$ hence $u(x)\equiv u(y)<0,\,x\in[0,1].$
                           However this contradicts the boundary condition $u(0)=u(1)=0.$

                       \end{proof}
                       \begin{rem} In Section ~\ref{seckernel} below we derive an expression for the resolvent kernel ~\eqref{eqkernelcont}. Since it is easy to see that the kernel is nonnegative, we obtain another proof of Proposition ~\ref{propmaximumcont}.
                       \end{rem}

                       \subsection{\textbf{POSITIVITY of the DISCRETE BIHARMONIC OPERATOR}}

                       We now show that the same positivity property holds also for the discrete biharmonic operator.
                       \begin{prop}\label{propmaximumdisc} Let
                       $$\delta^4_x\fu=\ff,$$
                       where $\fu,\,\fu_x\in l^2_{h,0}.$  Then the following comparison principle holds.

                       If $\ff_j\geq 0,\,\,0\leq j\leq N,$ then also $\fu_j\geq 0,\,\,0\leq j\leq N.$
                       \end{prop}
                       \begin{proof}
                       Suppose to the contrary that $\fu_{j_0}< 0$ for some index $1\leq j_0\leq N-1.$

                       Let $s_\fu\in C^2_0(\Omega)$ be the corresponding spline function. Since $s_\fu(x_{j_0})=\fu_{j_0}<0$ it follows that there exists a minimum point $y\in\Omega$ so that
                          $$s_\fu(y)=\min\set{s_\fu(x),\,\,x\in\Omega}<0.$$
                          We have
                          \be\label{eqsplineminy}
                          s_\fu'(y)=0,\,\,\,s_\fu''(y)\geq 0.
                          \ee
                          Since $s_\fu'$ vanishes at the endpoints, we infer that there are points
                         $$\xi\in (0,y),\,\,\eta\in (y,1),$$
                         such that
                         $$s_\fu''(\xi)=u''(\eta)=0.$$
                         Let
                         \be\label{eqdefabdisc}\aligned
                         a=\inf\Big\{\xi\in\Omega,\,\,s_\fu''(\xi)=0\Big\},\\
                         b=\sup \Big\{\eta\in\Omega,\,\,s_\fu''(\eta)=0\Big\}.
                         \endaligned\ee
                          Let $w(x)=s_\fu''(x).$ The function $w$ is continuous and linear in grid intervals. In view of Proposition ~\ref{cor3rdspljump} we get, in the sense of distributions,
                          \be
                            w''=h\suml_{j=1}^{N-1}\ff_j\delta_{x_j}\geq 0,
                          \ee
                          where $\delta_y$ is the Dirac measure at $y.$

                          Since $w(a)=w(b)=0,$ the standard maximum principle yields
                          $$w(x)\equiv 0,\,\,x\in[a,b],$$
                          hence
                          $$s_\fu'(x)\equiv s_\fu'(y)=0,\,\,x\in[a,b],$$
                          and in particular $s_\fu'(a)=s_\fu'(b)=0.$

                          As in the proof of Proposition ~\ref{propmaximumcont} we conclude that $a=0$ and $b=1,$ and therefore
                           $$s_\fu(x)\equiv s_\fu(y)<0,\,\,x\in[0,1],$$
                           which is a contradiction to the boundary conditions.

                       \end{proof}
                       \begin{cor} Let $\fu$ satisfy the conditions of Proposition ~\ref{propmaximumdisc}. Let $s_\fu$ be the corresponding spline function. Then
                       $$s_\fu(x)\geq 0,\quad x\in[0,1].$$
                       \end{cor}
                       \begin{proof}  The assumption that there exists a point $y\in(0,1)$ such that $s_\fu(y)<0$ leads to a contradiction; this follows from the proof of Proposition ~\ref{propmaximumdisc} .
                       \end{proof}
                       \section{\textbf{THE CONTINUOUS and DISCRETE RESOLVENT KERNEL}}\label{seckernel}

                        The operator $\Lcal=d^4/dx^4,$ with homogeneous boundary conditions ($\phi\in D(\Lcal)\Rightarrow \phi(0)=\phi'(0)=\phi(1)=\phi'(1)=0$) is positive definite (in particular self adjoint) with domain
        $D(\Lcal)= H^4([0,1])\cap H^2_0([0,1].$
                       We now consider the kernel of $\Lcal^{-1},$ namely,
  Green's function of the biharmonic problem
\be\label{eqcontu4f}\Lcal u=\Big(\parx\Big)^4u(x)=f(x),\ee
                       where $u\in H^4(\Omega)\cap H^2_0(\Omega).$
A standard computation leads to the following
\begin{claim}\label{greencont}
The solution of ~\eqref{eqcontu4f} is given by
\be\label{eqactionkernel}u(x)= \int_0^1 K(x,y)f(y)dy,\ee
where
\be\label{eqkernelcont} K(x,y)=\begin{cases}
\frac{1}{6}(1-x)^2y^2[2x(1-y)+ x-y], & y<x\\
\frac{1}{6}x^2(1-y)^2[2y(1-x)+y-x], & x<y
\end{cases}.\ee
\end{claim}
\begin{proof} By the general theory, we verify that in the sense of distributions, for each fixed $y,$ as a function of $x,$
$$\Big(\parx\Big)^4K(x,y)=\delta_y,$$
where $\delta_y$ is the Dirac measure at $y.$ In addition, $K(x,y)$ is symmetric in $x,\,y$ and satisfies the homogeneous boundary conditions (as a function of $x$).
\end{proof}
 \subsection{\textbf{EXTENDING the KERNEL to  $H^{-2}(\Omega)$}}

  The domain of $\Big(\parx\Big)^4$ (as a self-adjoint operator in $L^2(\Omega),$ subject to homogeneous boundary conditions) is $H^2_0(\Omega)\cap H^4(\Omega).$ When extended (in the sense of distributions) to $H^2_0(\Omega),$ it maps it to its dual $H^{-2}(\Omega)$ ~\cite[Chapter 5]{evans}. On the other hand, the general theory (or a direct inspection of the expression ~\eqref{eqkernelcont}) ensures that, for every fixed $x\in \Omega,$ we have $K(x,\cdot)\in H^2_0(\Omega).$ It follows that Equation ~\eqref{eqactionkernel} can be extended to all $u\in H^2_0(\Omega)$ (or, alternatively, to all $f\in H^{-2}(\Omega)$) as
  \be\label{eqactionkerdisc} u(x)=<K(x,y),f(y)>,
  \ee
  where $<\cdot,\cdot>$ is the $\Big(H^2_0(\Omega),\,H^{-2}(\Omega)\Big)$ coupling.

We now fix a mesh size $h=\frac1N$ and consider the grid functions
 $\fu\in l^2_{h,0}$  vanishing at the endpoints. As in Section ~\ref{secsplinehermit}
we let $s_\fu\in H^2_0(\Omega)$ be the corresponding spline function.

   Let
   $$SP_h=\set{s_\fu\in H^2_0(\Omega),\,\,\fu\in l^2_{h,0}}.$$
   We note that $SP_h$ is a finite-dimensional subspace of $H^2_0(\Omega).$  However, it is not fully contained
   in $H^4(\Omega).$ Therefore, as observed above, we can extend the differential operator $\Big(\parx\Big)^4$ to the union
   $ \Big[H^4(\Omega)\cap H^2_0(\Omega)\Big]\cup SP_h.$

   As was shown in Proposition ~\ref{cor3rdspljump}, the action of the operator on $SP_h$ is given by a combination of Dirac delta-functions at the nodes $x_j,$ that can be written as an equality of grid functions
    $$\Big(\parx\Big)^4s_\fu= h\delta^4_x\fu.$$
    The right-hand side in this equation is a finite measure, and we recall that, owing to the Sobolev embedding theorem,  all finite measures are contained in $H^{-2}(\Omega).$

   Thus, Equation ~\eqref{eqactionkerdisc} takes here the form
     \be\label{eqschwartzkernel}\fu_j=h\suml_{i=1}^{N-1}K(x_i,x_j)(\delta^4_x\fu)_i,\quad j=1,2,\ldots,N-1.\ee
     \begin{cor}\label{cordiscresolvent} The discrete operator  $(\delta^4_x)^{-1}: l^2_{h,0}\to l^2_{h,0}$
         is represented by a matrix $\set{K^h_{i,j}}_{1\leq i,j\leq N-1},$   explicitly given by
          \be\label{eqKheqK} K^h_{i,j}=h K(x_i,x_j),\quad 1\leq i,j\leq N-1,\ee
          where $K(x,y)$ is the resolvent kernel of $\Big(\parx\Big)^4,$ as  in Equation ~\eqref{eqkernelcont}.
     \end{cor}

\section{\textbf{CONTINUOUS and DISCRETE EIGENVALUES}}\label{seceigenvalues}
\subsection{\textbf{THE CONTINUOUS OPERATOR}}
  We now consider the eigenvalues of the  operator $\Lcal,$ introduced in Section ~\ref{seckernel}.

        The operator has a compact resolvent, and the kernel $K$ of $\Lcal^{-1}$ is given in Claim ~\ref{greencont}.  The spectrum
        of $\Lcal$ consists of an
        increasing sequence of positive simple eigenvalues, which we designate as  $\set{0<\lambda_1<\lambda_2<...<\lambda_k<...}.$

        Since these eigenvalues play an important role in the sequel, we provide below the details of their evaluation, repeating the proof of ~\cite[Lemma 5.5.4]{davies}.

       Let $\phi\in H^4([0,1])\cap H^2_0([0,1])$ be a real eigenfunction
         $$\frac{d^4}{dx^4}\phi=\lambda\phi,\quad \lambda\in \set{0<\lambda_1\leq...\leq\lambda_k...}.$$

         Clearly, this function must be of the form
         \be\label{eigenf4}
         \phi(x)=A\cos(\beta x)+B\sin(\beta x)+C\cosh(\beta x)+D\sinh(\beta x),
         \ee
         where $\beta$ is real and $\beta^4=\lambda.$

         The conditions $\phi(0)=\phi'(0)=0$ clearly imply
         $$A=-C,\quad B=-D,$$
         and $\phi(1)=0$ yields
         \be\label{ABco}A(\cos\beta -\cosh\beta )=-B(\sin\beta -\sinh\beta ).\ee
          The remaining condition $\phi'(1)=0$ yields
           $$-B(\cos\beta -\cosh\beta )=A(-\sin\beta -\sinh\beta ).$$
           Multiplying the two equations and invoking standard identities we get
           \be\label{betaco}
           \cos\beta\cosh\beta=1,
           \ee
            which is to be considered as the equation determining the discrete eigenvalues.

            Changing $\beta\to -\beta$ we can keep $A,C$ unmodified but reverse the signs of $B,D.$
            It therefore follows that for $-\beta<0$ (solution of ~\eqref{betaco}) we get the same eigenfunction
            ~\eqref{eigenf4} as for $\beta>0,$ and we can consider only positive $\beta.$

            We therefore get the full set of eigenfunctions (for $\beta>0$ solving ~\eqref{betaco}),
            \be\label{twoeigenb}
            \phi(x)=A\cos(\beta x)+B\sin(\beta x)-A\cosh(\beta x)-B\sinh(\beta x),
                        \ee
                        where $A,B$ satisfy ~\eqref{ABco}.

            In order to estimate the location of the eigenvalues it therefore suffices to consider the positive solutions of ~\eqref{betaco}. The following claim is easy to verify.
            \begin{claim}\label{claimcontevs} Equation ~\eqref{betaco} has a sequence of positive solutions as follows.
              \be \label{allbeta}\begin{cases}
               \beta_0\in (3\pi/2,2\pi),\\
               \beta^{(1)}_k\in (2k\pi,(2k+1/2)\pi),\quad k=1,2,...\\
               \beta^{(2)}_k\in ((2k+3/2)\pi,(2(k+1)\pi),\quad k=1,2,...
               \end{cases}\ee
               The corresponding eigenvalues $\lambda_0=\beta_0^4,\,\lambda^{(1)}_k=(\beta^{(1)}_k)^4,\,\lambda^{(2)}_k=(\beta^{(2)}_k)^4$ of $\Lcal$ are all simple.
            \end{claim}
               We denote by
              $$ \set{\phi_1,...,\phi_k...}$$
              the orthonormal set of the associated eigenfunctions.

              \subsection{\textbf{THE DISCRETE OPERATOR}}

We simplify the notation above and denote by $\set{0<\lambda_1<\lambda_2<\ldots<\lambda_k<\ldots} $  the (infinite) sequence of eigenvalues of
$\Lcal=\Big(\parx\Big)^4.$

 Given $h=\frac1N,$ let $$\Lambda_h=\set{0<\lambda_{h,1}\leq\lambda_{h,2}\leq\ldots\leq\lambda_{h,N-1}} $$ be the finite sequence of eigenvalues of $\delta^4_x.$

 We denote by $\Gamma$ the sum
   $$\Gamma=\suml_{i=1}^\infty \lambda_i^{-1},$$
   and let
   $$\Gamma_h=\suml_{i=1}^{N-1} \lambda_{h,i}^{-1}.$$
   \begin{prop}There exists a constant $C>0,$ independent of $h,$ so that
   \be\label{eqestimgammaGammah}
           |\Gamma-\Gamma_h|\leq Ch^4.\ee
   \end{prop}
   \begin{proof}We introduce the (infinite) set of reciprocals of the eigenvalues of $\Lcal,$
   namely, the eigenvalues of the kernel $K(x,y)$ ~\eqref{eqkernelcont},

   \be\label{eqLmbdamin1}\Lambda^{-1}=\set{\lambda_{1}^{-1}>\lambda_{2}^{-1}>\ldots>\lambda_{k}^{-1}\ldots >0} , \ee
                   while
                     \be\label{eqLmbdahmin1}\Lambda_h^{-1}=
                     \set{\lambda_{h,1}^{-1}\geq\lambda_{h,2}^{-1}\geq\ldots\geq\lambda_{h,N-1}^{-1}>0} \ee
                    is the set of eigenvalues of $(\delta^4_x)^{-1},$ corresponding to the discrete kernel $K^h$ ~\eqref{eqKheqK}.

   By the standard trace formula, it follows that
   \be \Gamma=\int_0^1K(x,x)dx,\quad \Gamma_h=h\suml_{i=1}^{N-1}K(x_i,x_i).
   \ee
    Since $K(x,x)=\frac13 x^3(1-x)^3,$ the numerical values of $\Gamma$ and $C$ can easily be calculated, and it turns out that
    \be  \Gamma=\frac{1}{420}.
    \ee
    On the other hand
    \be
    \Gamma_h=\frac{h}{3}\suml_{i=1}^{N-1}(ih)^3(1-ih)^3=\frac{1}{420}+\frac{1}{180}h^4-\frac{1}{126}h^6,
    \ee
    so that ~\eqref{eqestimgammaGammah} is established (and even with an explicit constant).
   \end{proof}
   \begin{rem}

       Observe that $\Gamma_h$ is the discrete trapezoidal approximation to the integral for $\Gamma.$
       By the standard estimate for the trapezoidal rule, we obtain
       \be\label{eqestimgammaGammah2}
   |\Gamma-\Gamma_h|\leq Ch^2,
   \ee
        with $C=\frac{1}{12}\max\limits_{0\leq x\leq 1}|(\parx)^2K(x,x)|=\frac{1}{96}.$

        The fourth-order estimate ~\eqref{eqestimgammaGammah} is clearly a result of a closer inspection of the kernel $K.$
   \end{rem}
 The ``collective'' estimate ~\eqref{eqestimgammaGammah} does not imply that an estimate of the form $\lambda_i^{-1}-\lambda_{h,i}^{-1}=O(h^4)$ is valid, for any fixed value of the index $i.$ However, the next proposition provides a weaker statement in this direction. It will play a key role in the final, stronger Theorem ~\ref{thmoptimalallevs} below.
         \begin{prop}\label{propdistLambdah} For any fixed integer $i\geq 1$ there exist positive constants $C,\,h_0>0$ such that for any $0<h=\frac1N<h_0$ we have
         \be\label{eqdistlambdai}dist \{\lambda_i^{-1},\,\Lambda_h^{-1}\}\leq Ch^4,\ee
         where $\Lambda_h^{-1}$ is the set of reciprocals introduced in ~\eqref{eqLmbdahmin1}.
         \end{prop}
          \begin{proof} Let $\phi_i(x)\in H^2_0(\Omega)$ be a normalized eigenfunction of $\Big(\parx\Big)^4,$ corresponding to
          $\lambda_i.$ Recall that $\phi_i\in C^\infty$  and $\Big(\parx\Big)^{-4}\phi_i=\lambda_i^{-1}\phi_i.$ Hence
          $$\lambda_i^{-1}\phi_i(x)=\int_0^1K(x,y)\phi_i(y)dy,\quad x\in\Omega.$$

            For simplicity, we denote by $\set{x_j=jh,\,0\leq j\leq N}$ the grid points , omitting the obvious dependence on $h.$

          Let $\phi_i^\ast=\set{\phi_i(x_0),\ldots,\phi_i(x_k),\ldots,\phi_i(x_N)}$ be the corresponding grid function.

          In view of Claim ~\ref{claim21} and Corollary ~\ref{cordiscresolvent} we have for all $0\leq k\leq N,$
          $$\Big|\lambda_i^{-1}\phi_i(x_k)-h\suml_{j=0}^N K(x_k,x_j)\phi_i(x_j)\Big|\leq Ch^4,$$
          where here and below $C>0$ is a constant depending only on $\phi_i$ that changes from one estimate to the next.  Using the notation ~\eqref{eqKheqK} this can be rewritten as
          \be\label{eqdiffKhviast}
          \Big|\lambda_i^{-1}\phi_i^\ast(x_k)-\suml_{j=0}^N K^h_{k,j}\phi_i^\ast(x_j)\Big|\leq Ch^4,
          \ee
           that is
           $$\Big|(\lambda_i^{-1}-(\delta_x^4)^{-1})\phi_i^\ast\Big|_h\leq Ch^4.$$
           On the other hand, the smoothness of the normalized $\phi_i$ yields
           $$|\phi_i^\ast|_h\geq 1-Ch.$$
           The last two estimates imply the following estimate of the operator norm,
           \be\label{eqopernormKh}
           \Big|\Big(\lambda_i^{-1}-(\delta_x^4)^{-1}\Big)^{-1}\Big|_h\geq \frac{1-Ch}{Ch^4}\geq Ch^{-4},\ee
           for $h<h_0.$ By a standard result concerning resolvents of self-adjoint operators we conclude that
           $$dist \{\lambda_i^{-1},\,\Lambda_h^{-1}\}=\Big|\Big(\lambda_i^{-1}-(\delta_x^4)^{-1}\Big)^{-1}\Big|_h^{-1},$$
           which concludes the proof of the proposition.
          \end{proof}
          \begin{rem} Proposition ~\ref{propdistLambdah} shows that in any neighborhood of $\lambda_i^{-1}$ there is a discrete eigenvalue $\lambda_{h,k}^{-1},$ provided $h>0$ is sufficiently small. Observe, however, that we cannot infer that, even  the largest eigenvalue (of $\Lcal^{-1}$) $\lambda_1^{-1}$ is the limit, as $h\to 0,$ of the largest discrete eigenvalue $\lambda_{h,1}^{-1}$ (of ($\delta_x^4)^{-1}$).   This is done in Theorem ~\ref{thmconvergelambda} below.
          \end{rem}

          \begin{rem}\label{remnystrom}
          In view of Corollary ~\ref{cordiscresolvent} the discrete eigenvalues in $\Lambda_h^{-1}$ are obtained by a ``Nystr\"{o}m method''~\cite{nystrom}, namely, eigenvalues of the discretized kernel. The fact that  for any fixed integer $i\geq 1$
         $$\lim\limits_{h\to 0}dist \{\lambda_i^{-1},\Lambda_h^{-1}\}=0,$$
         follows from ~\cite[Theorem 3]{nystrom}.   Proposition ~\ref{propdistLambdah} establishes an ``optimal'' $O(h^4)$ rate to this convergence.
          \end{rem}
          \subsection{\textbf{CONVERGENCE OF THE FIRST DISCRETE EIGENVALUE}}

          For the first discrete eigenvalue $\lambda_{h,1}$ we can establish its convergence (as $h\downarrow 0$) to $\lambda_1$ as follows.
      \begin{thm}\label{thmconvergelambda} The sequence of the discrete first eigenvalues of $\delta_x^4$ converges to the first eigenvalue of the continuous operator $\Lcal:$

      \be\label{eqlimlambda1}
      \lim\limits_{h\to 0}\lambda_{h,1}=\lambda_1.
      \ee

      \end{thm}
      \begin{proof}
      We prove in fact that
      \be\label{eqdeflambda01}\lim\limits_{h\to 0}\lambda_{h,1}^{-1}=\lambda_1^{-1}.\ee

      We first prove  that
     \be\label{eqliminflambdah} \liminf\limits_{h\to 0}\lambda_{h,1}^{-1}\geq \lambda_1^{-1}.\ee

       Given $\eps>0,$ it suffices to prove that there exists $h_0>0$ so that for any $0<h<h_0,$
      \be\label{eqcomparelambda}
        \lambda_{h,1}^{-1}\geq\lambda_1^{-1}-\eps.
         \ee
         Since $\lambda_1^{-1}$ is the greatest eigenvalue of the kernel $K,$ we have
            \be\label{eqlambdamaxcont}\lambda_1^{-1}=\max\limits_{\|u\|_{L^2(0,1)}=1}\int_0^1\int_0^1K(x,y)u(x)u(y)dxdy.\ee

            Remark that (see the proof of Proposition ~\ref{propdistLambdah}) the maximum is attained by $\phi_1,$  the normalized eigenfunction corresponding to $\lambda_1.$ However we shall need an approximating compactly supported function.

            Now let $ u^\eps\in C^\infty_0(0,1) $ be a normalized function , $\|u^\eps\|_{L^2(0,1)}=1$ and such that
            \be\label{eqestimlambda1}\lambda_1^{-1}-\eps\leq\int_0^1\int_0^1K(x,y)u^\eps(x)u^\eps(y)dxdy.\ee
            Take $h_0>0$ sufficiently small, so that $u^\eps$ vanishes in a neighborhood of the  ``edge'' intervals $[0,h_0]\cup
            [1-h_0,1].$

            Let $h=\frac1N<h_0.$

            For simplicity, we denote by $\set{x_j=jh,\,0\leq j\leq N}$ the grid points , omitting the obvious dependence on $h.$

            Define a nonnegative step function
            $$U^\eps(x)^2=\frac{1}{h}\int_{x_j-\frac{h}{2}}^{x_j+\frac{h}{2}}u^\eps(x)^2dx,\quad
            x\in\Big(x_j-\frac{h}{2},x_j+\frac{h}{2}\Big),\,\,j=1,2,\ldots,N-1.$$
            Clearly $\|U^\eps\|_{L^2(0,1)}=1.$

            The continuity of $K(x,y)$ implies that (decreasing $h_0$ if necessary)
            \be\label{eqdiscreteK}
            \int_0^1\int_0^1K(x,y)u^\eps(x)u^\eps(y)dxdy\leq h^2\suml_{i,j=1}^{N-1}K(x_i,x_j)U^\eps(x_i)U^\eps(x_j)+\eps.
            \ee

            Let $\fu^\eps=(\fu^\eps_1,\ldots,\fu^\eps_{N-1})\in l^2_{h,0}$ be the grid function defined by
            $$\fu^\eps_j=U^\eps(x_j),\quad j=1,2,\ldots,N-1,$$
            so that $|\fu^\eps|_{h}=1.$

            Employing the notation ~\eqref{eqKheqK}, the  inequality ~\eqref{eqdiscreteK} can be rewritten as
            \be \label{eqdiscreteKueps}\int_0^1\int_0^1K(x,y)u^\eps(x)u^\eps(y)dxdy\leq h\suml_{i,j=1}^{N-1}K^{h}(x_i,x_j)\fu^\eps_i\fu^\eps_j+\eps=
            (K^{h}\fu^\eps,\fu^\eps)_{h}+\eps.
            \ee

            From the maximum principle
                        (see the notation introduced in Corollary ~\ref{cordiscresolvent}),
           \be\label{eqlambdahinv} \lambda_{h,1}^{-1}=\max\limits_{|\fu|_{h,0}=1}((\delta_x^4)^{-1}\fu,\fu)_h=
            h\max\limits_{|\fu|_{h,0}=1}\suml_{i,j=1}^{N-1}K^h_{i,j}\fu_i\fu_j.\ee
             we infer that
            \be\label{eqestimKdislambdah}
            h\suml_{i,j=1}^{N-1}K^{h}(x_i,x_j)\fu^\eps_i\fu^\eps_j\leq \lambda_{h,1}^{-1}.
            \ee
            Combining ~\eqref{eqestimlambda1},~\eqref{eqdiscreteKueps} and ~\eqref{eqestimKdislambdah} we obtain
            \be \lambda_1^{-1}\leq \lambda_{h,1}^{-1}+2\eps.
            \ee
            The estimate ~\eqref{eqliminflambdah} is therefore established.

            We now proceed to establish the reverse inequality
            \be\label{eqlimsuplambdah} \limsup\limits_{h\to 0}\lambda_{h,1}^{-1}\leq \lambda_1^{-1}.\ee
            Given $\eps>0,$ it suffices to prove that there exists $h_0>0$ so that for any $0<h<h_0,$
      \be\label{eqcomparelambdasup}
        \lambda_{h,1}^{-1}\leq\lambda_1^{-1}+\eps.
         \ee
         Let $\fu^h\in l^2_{h,0},\,\,|\fu^h|_{h,0}=1,$ be an eigenvector corresponding to $\lambda_{h,1}$, so that
         \be\lambda_{h,1}^{-1}=
            h\suml_{i,j=1}^{N-1}K^h_{i,j}\fu^h_i\fu^h_j.\ee
            Since the kernel $K^h$ is positive, we can assume that $\fu^h_i\geq 0,\,\,0\leq i\leq N.$

            Let $u^h(x)$ be the nonnegative piecewise constant function defined by
\be
            u^h(x)=\fu^h_i,\quad x_i-\frac{h}{2}\leq x\leq x_{i}+\frac{h}{2},\,\,i=0,1,\ldots,N.
      \ee

            Clearly $\|u^h\|_{L^2(0,1)}=1$ so in view of ~\eqref{eqlambdamaxcont}
            \be\label{eqlambdamin1ueps} \lambda_1^{-1}\geq \int_0^1\int_0^1K(x,y)u^h(x)u^h(y)dxdy.
            \ee
            We now replace the kernel $K(x,y)$ by the piecewise constant kernel
            \be\label{eqDefineKh}
               K_h(x,y)=K(x_i,y_j),\quad x\in \Big(x_i-\frac{h}{2},x_i+\frac{h}{2}\Big),\,\,y\in \Big(y_j-\frac{h}{2},y_j+\frac{h}{2}\Big),\,\,0\leq i,j\leq N.
            \ee
             By increasing $N$ if needed, the continuity of $K(x,y)$ implies that
             $$\int_0^1\int_0^1 |K(x,y)-K_h(x,y)|^2dxdy\leq \eps^2,$$
            so that, by the Cauchy-Schwarz inequality,
            \be\label{eqcomparekKtild}
            \Big|\int_0^1\int_0^1K(x,y)u^h(x)u^h(y)dxdy -\int_0^1\int_0^1K_h(x,y)u^h(x)u^h(y)dxdy\Big|\leq \eps.
            \ee
            Observe that when changing $N$ we must also change  $\fu^h$ (hence $u^h$), but since they are normalized this change does not affect the above estimate.

            Combining ~\eqref{eqlambdamin1ueps} and ~\eqref{eqcomparekKtild} we obtain
            \be
            \lambda_1^{-1}\geq \int_0^1\int_0^1K_h(x,y)u^h(x)u^h(y)dxdy-\eps.
            \ee
            Now
            \be\aligned
            \int_0^1\int_0^1K_h(x,y)u^h(x)u^h(y)dxdy=\suml_{i,j=0}^N K(x_i,y_j)\int_{x_i-\frac{h}{2}}^{x_i+\frac{h}{2}}u^h(x)dx\cdot
            \int_{y_j-\frac{h}{2}}^{y_j+\frac{h}{2}}u^h(y)dy\\
            =h^2\suml_{i,j=0}^N K(x_i,y_j)\fu^h_i\fu^h_j= h\suml_{i,j=1}^{N-1}K^h_{i,j}\fu^h_i\fu^h_j=\lambda_{h,1}^{-1}.
            \endaligned\ee
            Thus ~\eqref{eqcomparelambdasup} is established and the proof is complete.

      \end{proof}
          Theorem ~\ref{thmconvergelambda} does not give any convergence rate for the difference $|\lambda_1-\lambda_{h,1}|.$  In what follows we consider this issue, using the basic variational tools.

          We begin with a more general discussion.

       Pick $\phi\in  \set{\phi_1,...,\phi_k...}$ a normalized eigenfunction of $\Lcal,$ with associated eigenvalue $\lambda\in \set{0<\lambda_1<\lambda_2<\ldots<\lambda_k<\ldots}.$

       Applying the operator $\Lcal$ to

      $$\Lcal\phi=\lambda\phi,\quad \lambda\in \set{0<\lambda_1<...<\lambda_k...}..,$$

       we get

       $$\frac{d^8}{dx^8}\phi=\lambda^2\phi.$$

       Since $\phi$ is normalized, we have
       \be \label{derphibd}\|\frac{d^8}{dx^8}\phi\|_{L^2[0,1]}=\lambda^2,
       \ee
       and continuing in this fashion we see that all derivatives of $\phi$ are bounded by some power of $\lambda,$
       and therefore in the estimates below we have a generic constant $C>0$ depending only on $\lambda.$

            Let $\phi^*$ be
           the corresponding grid function, $\phi^*(x_i)=\phi(x_i),\quad 0\leq i\leq N.$

         Let $\fv\in l^2_{h,0}$ satisfy
         $$\delta_x^4\fv=\lambda\phi^*,$$
         where also $\fv_x\in l^2_{h,0}.$

         By the fourth order accuracy ~\eqref{eqoptimalerror}  we know

         \be\label{vphi}|\fv-\phi^*|_\infty\leq Ch^4,\ee
         where $C$ is independent of $N=h^{-1},$ but depends of course on $\phi.$

         It follows that
          \be\label{vphi1}\delta_x^4\fv=\lambda\fv+\fw ,\,\,|\fw|_h\leq Ch^4.\ee

          Since $\phi$ is normalized, the truncation error for the trapezoid integration gives
          \be |\phi^*|_h^2=h\suml_{i=1}^{N-1}[\phi^*_i]^2=\|\phi\|_{L^2[0,1]}^2+O(h^2)=1+O(h^2),
          \ee
          hence also
          \be
          |1- |\fv|_h^2|\leq Ch^2.
          \ee
           Let $\bar{\fv}=\frac{\fv}{|\fv|_h},$ then it follows from ~\eqref{vphi1}

        \be\label{vphi2}\delta_x^4\bar{\fv}=\lambda\bar{\fv}+\bar{\fw} ,\,\,|\bar{\fw}|_h\leq Ch^4.\ee

         Regarding the first eigenvalue, we can now show that $\lambda_{h,1}$ can exceed $\lambda_1$ by at most $O(h^4).$
        \begin{claim}\label{claimuplimith4} Let $\lambda_1$ be the first eigenvalue of $\Lcal$ ( by ~\eqref{allbeta}, $\lambda_1=\beta_0^4$).
             Then there exists a constant $C>0,$ depending on the eigenfunction $\phi_1,$ but not on $h,$ such that
              \be\label{lambdamu1} \lambda_{h,1}\leq\lambda_1+Ch^4.\ee
           \end{claim}
           \begin{proof}
           Consider ~\eqref{vphi2} with $\lambda=\lambda_1.$ By the variational minimum principle for the first eigenvalue
           we know that
           $$\lambda_{h,1}=\min\limits_{|\fz|_h=1}(\delta_x^4\fz,\fz)_h,$$
           hence
           \be\label{eqestimatefirstev}
           \lambda_{h,1}\leq (\delta_x^4\bar{\fv},\bar{\fv})_h\leq \lambda_1+Ch^4,
           \ee
            which proves the claim.

             \end{proof}

             \begin{rem}  The exact first eigenvalue is $\lambda_1= 500.5639017404. $ Numerical calculations actually show that $\lambda_{h,1}\leq\lambda_1,$ and that $\lambda_{h,1}$ increases as $h$ decreases. This is shown in Figure ~\ref{figfirstevhaggai} . We are still  unable to prove this monotonicity.

             \end{rem}
             \begin{rem}\label{remestfirstev}
             Observe that in Claim ~\ref{claimuplimith4} we do not have a corresponding lower limit, namely, that $\lambda_{h,1}$ is above  $\lambda_1-O(h^4).$  This is evident in the numerical results displayed in Figure
             ~\ref{figfirsteverrhaggai}. The proof of this fact is postponed to Theorem ~\ref{thmoptimalallevs} below, where we show that the convergence of all discrete eigenvalues
             to the corresponding continuous ones is ``optimal'', namely, at an $O(h^4)$ rate.
             \end{rem}
             \begin{figure}[htb]
\centering
\includegraphics[width=150pt,height=150pt]{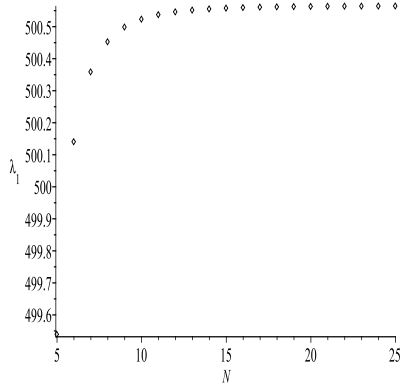}

\caption{First discrete eigenvalue as a function of the number of grid points in $[0,1].$
}
\label{figfirstevhaggai}
\end{figure}

\begin{figure}[htb]
\centering
\includegraphics[width=150pt,height=150pt]{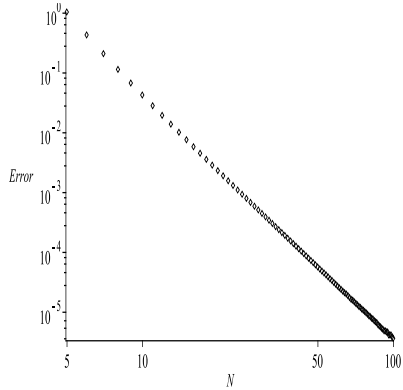}

\caption{Log-log graph of the error of first discrete eigenvalue $\lambda_1-\lambda_{h,1}$ as function of the number $N$ of grid points in $[0,1].$ The slope is -4, indicating a convergence rate $O(N^{-4})=O(h^4).$
}
\label{figfirsteverrhaggai}
\end{figure}


             \newpage

           \subsection{\textbf{CONVERGENCE OF THE  DISCRETE EIGENVALUES $\lambda_{h,k},\,\,k>1$}}
           $$ $$
           We now consider the convergence of all discrete eigenvalues to their continuous counterparts.

Numerical simulations indicate that, if we \textit{fix an index $k,$ } then
                $$  |\lambda_{k}-\lambda_{h,k}|\leq Ch^4,\quad \mbox{as} \,\,h\to 0, $$
                with $C>0$ depending on $k.$ This is demonstrated in Figure ~\ref{figevs16} (for $N=16$) and  Figure ~\ref{figevs64}  (for $N=64$). We thank Jean-Pierre Croisille for both figures. Thus, even  a very coarse resolution produces excellent approximation of the eigenvalues.

             \begin{figure}[htb]
\centering
\includegraphics[width=150pt,height=150pt]{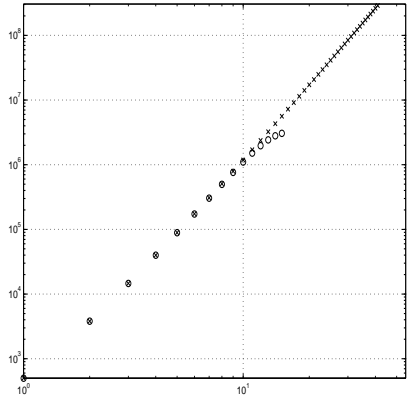}
\caption{Log-log graph of eigenvalues: Horizontal-$\log k.$ Vertical --$\log \lambda_k \,(\times),\,\log \lambda_{h,k}\,(\circ),\,h=\frac{1}{N}=\frac{1}{16}$
.}
\label{figevs16}
\end{figure}
               \begin{figure}[htb]
\centering
\includegraphics[width=150pt,height=150pt]{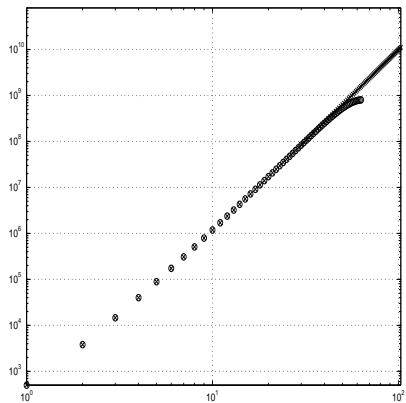}

\caption{Log-log graph of eigenvalues: Horizontal-$\log k.$ Vertical --$\log \lambda_k \,(\times),\,\log \lambda_{h,k}\,(\circ),\,h=\frac{1}{N}=\frac{1}{64}$
.}
\label{figevs64}
\end{figure}

    The convergence result in Theorem ~\ref{thmconvergelambda}, that dealt with the first eigenvalue, did not yield an ``optimal'' convergence rate, as noted in Remark ~\ref{remestfirstev}.

      Using a very different approach, we shall now extend the convergence to all eigenvalues, and, furthermore, obtain the optimal $O(h^4)$ convergence rate.

       Let $K_h(x,y)$ be the piecewise constant (positive definite) kernel introduced in ~\eqref{eqDefineKh}. We denote by $\Lcal_h^{-1}$ the operator (on $L^2[0,1]$) whose kernel is $K_h.$ Clearly this operator is compact and positive definite. In fact, the following claim asserts that it has only finitely many positive eigenvalues (depending on $h,$ of course).
       \begin{claim} The set of eigenvalues of $\Lcal_h^{-1}$ is  the finite set $\Lambda_h^{-1},$ defined in ~\eqref{eqLmbdahmin1}.
       \end{claim}
       \begin{proof}
        Let $u\in L^2[0,1]$ be an eigenfunction of $\Lcal_h^{-1}.$ Thus, for some $\mu>0,$
        $$\mu u(x)=\int\limits_0^1K_h(x,y)u(y)dy,\quad x\in [0,1].$$
        In particular, $u$ is piecewise constant
         $$u(x)=u(x_i),\quad x\in \Big(x_i-\frac{h}{2},x_i+\frac{h}{2}\Big),\,\,i=0,1,\ldots,i=N.$$
         hence (with $K^h$ as in Corollary ~\ref{cordiscresolvent})
         \be
          \mu u(x_i)=\suml_{j=0}^N K^h_{i,j}u(x_j),\quad 0\leq i\leq N,
         \ee
         where the boundary values $u(x_0)=u(x_N)=0$ are included.

         Thus $\mu$ is an eigenvalue of $(\delta_x^4)^{-1},$ hence $\mu=\lambda_{h,k}^{-1}$ for some $1\leq k\leq N-1.$
       \end{proof}
            We now proceed to establish the convergence of all discrete eigenvalues to the corresponding continuous ones. In fact, the following lemma is a special case of a theorem of Markus ~\cite[Corollary 5.3]{markus} concerning differences of eigenvalues  of self-adjoint operators. A similar general theorem was proved (much later) by Kato~\cite{kato}. However the generality of Kato's theorem required an ``extended enumeration'' of the eigenvalues, adding values of boundary points of the essential spectra.

             For the convenience of the reader we provide here a simple proof of the lemma, following  the proof of (the finite-dimensional) Theorem 6.11 in ~\cite[Section II.6]{kato-book}.

             \begin{lem}\label{lemconvergence} Let $h=\frac{1}{N}, $ and let
             $$\Lambda^{-1}=\set{\lambda_{1}^{-1}>\lambda_{2}^{-1}>\ldots>\lambda_{k}^{-1}\ldots >0} , $$

                     $$\Lambda_h^{-1}=
                     \set{\lambda_{h,1}^{-1}\geq\lambda_{h,2}^{-1}\geq\ldots\geq\lambda_{h,N-1}^{-1}>0} ,$$
                     be the sets introduced in ~\eqref{eqLmbdamin1},~\eqref{eqLmbdahmin1}, respectively.

                     Then there exists a constant $C>0,$ independent of $h,$ so that
                     \be\label{eqthmsquaredifference}
                     \suml_{k=1}^{N-1}|\lambda_{k}^{-1}-\lambda_{h,k}^{-1}|^2+\suml_{k=N}^\infty\lambda_{k}^{-2}\leq\int_0^1\int_0^1 |K(x,y)-K_h(x,y)|^2dxdy\leq Ch^2.
                     \ee

             \end{lem}

              \begin{proof}
                Note that both $\Lcal^{-1},\,\Lcal_h^{-1},$ are Hilbert-Schmidt (hence compact) positive  operators.

                For $t\in[0,1]$ let
                $$\Lcal^{-1}_{t,h}=(1-t)\Lcal^{-1}+t\Lcal_h^{-1},$$
                which is also compact, positive self-adjoint operator. In particular, its spectrum (apart from $0$) consists of a descending sequence of  positive eigenvalues
                $$\set{\mu_{1}^{-1}(t)\geq\mu_{2}^{-1}(t)\geq\ldots\geq\mu_{N-1}^{-1}(t)\geq\mu_{N}^{-1}(t) \geq\ldots\mu_{N+p}^{-1}(t)\geq\ldots>0 },\,0\leq t\leq 1.$$
                In view of the discussion in ~\cite[Chapter VII.3.2]{kato-book} the functions $\mu_{k}^{-1}(t),\,1\leq k<\infty,$ are  continuous, piecewise analytic functions of $t,$ and satisfy
                \be\label{eqmuk0}\mu_{k}^{-1}(0)=\lambda_{k}^{-1},\,1\leq k<\infty,\ee
                and
                \be\label{eqmuk1}\mu_{k}^{-1}(1)=\begin{cases}\lambda_{h,k}^{-1},\,1\leq k<N,\\
                0,\,k\geq N.\end{cases}\ee

                In addition, there exists (for every fixed $t\in[0,1]$) a corresponding set of orthonormal functions (in $L^2(0,1)$)
                $$\set{\phi_1(x;t),\phi_{2}(x;t),\ldots,\phi_{N}(x;t),\ldots,\phi_k(x;t),\ldots},\,\,0\leq t\leq 1.$$
                Pick an index $k\geq 1.$ The eigenvalue $\mu_{k}^{-1}(t)$ is continuous (in $t\in [0,1]$) and piecewise analytic, with finitely many singularities. The associated eigenfunction
                $\phi_k(x;t)$ is piecewise analytic in $t,$  with the same (finitely many) singularities. Thus, the equation
                \be\label{eqeigenLht}
                \Big[(1-t)\Lcal^{-1}+t\Lcal_h^{-1}-\mu_{k}^{-1}(t)\Big]\phi_k(x;t)=0
                \ee
                can be differentiated with respect to $t$ (excluding the singularities) and we obtain
                \be\label{eqeigenLhtdif}
                \Big[\Lcal_h^{-1}-\Lcal^{-1}-\partt\mu_{k}^{-1}(t)\Big]\phi_k(x;t)+\Big[(1-t)\Lcal^{-1}+t\Lcal_h^{-1}-\mu_{k}^{-1}(t)\Big]\partt\phi_k(x;t)=0.
                \ee
                Taking the scalar product with $\phi_k(x;t)$ we conclude that
                \be\label{eqderivmukt}
                \partt\mu_{k}^{-1}(t)=\Big((\Lcal_h^{-1}-\Lcal^{-1})\phi_k(x;t),\phi_k(x;t)\Big)_{L^2(0,1)},\quad t\in[0,1].
                \ee
                Integrating this equation and taking ~\eqref{eqmuk0} and ~\eqref{eqmuk1} into account we get
                \be\label{eqderivmuktint}
               \int_0^1 \Big((\Lcal_h^{-1}-\Lcal^{-1})\phi_k(x;t),\phi_k(x;t)\Big)_{L^2(0,1)}dt=\begin{cases}\lambda_{h,k}^{-1}-\lambda_{k}^{-1},\,1\leq k<N,\\ \\
               -\lambda_{k}^{-1},\, k\geq N
                 .\end{cases}\ee

                The self-adjoint operator $\Acal=\Lcal_h^{-1}-\Lcal^{-1}$ is Hilbert-Schmidt, hence compact. Let $\set{\gamma_1,\gamma_2,\ldots}$ be the sequence of its non-zero eigenvalues (repeated according to multiplicity) with a corresponding orthonormal sequence of eigenfunctions $\set{\chi_1(x),\chi_2(x),\ldots}\subseteq L^2(0,1).$

                Since $\phi_k(x;t)=\suml_{j=1}^\infty (\phi_k(x;t),\chi_j(x))_{L^2(0,1)}\chi_j(x),$ Equation ~\eqref{eqderivmuktint} entails
               \be\label{eqgammajsigmajk} \suml_{j=1}^\infty \sigma_{j,k}\gamma_j=\begin{cases}\lambda_{h,k}^{-1}-\lambda_{k}^{-1},\,1\leq k<N,\\ \\
               -\lambda_{k}^{-1},\, k\geq N
                 ,\end{cases}\ee
                 where $\sigma_{j,k}=\int_0^1(\phi_k(x;t),\chi_j)_{L^2(0,1)}^2dt,\,\,1\leq j,k<\infty.$

                 By the orthonormality of the functions (in $x$)
                 $$ 0\leq \sigma_{j,k}\leq 1,\quad \suml_{j=1}^\infty\sigma_{j,k}\leq 1,\quad \suml_{k=1}^\infty\sigma_{j,k}\leq 1.
                 $$
                 Let $\Phi$ be a real convex function on the real line, with $\Phi(0)=0.$ From Jensen's inequality we get
                 $$\Phi\Big( \suml_{j=1}^\infty \sigma_{j,k}\gamma_j\Big)\leq \suml_{j=1}^\infty \sigma_{j,k}\Phi(\gamma_j),\quad k=1,2,\ldots,$$
                 and summation over $k$ yields
                 \be\label{eqPhigammaj}
                 \suml_{k=1}^\infty\Phi\Big( \suml_{j=1}^\infty \sigma_{j,k}\gamma_j\Big)\leq \suml_{j=1}^\infty \Phi(\gamma_j).
                 \ee
                 In particular , taking $\Phi(\xi)=\xi^2$ and noting ~\eqref{eqgammajsigmajk} we obtain
                 $$\suml_{k=1}^{N-1}|\lambda_{k}^{-1}-\lambda_{h,k}^{-1}|^2+\suml_{k=N}^\infty\lambda_{k}^{-2}\leq \suml_{j=1}^\infty\gamma_j^2.$$
                 The sum on the right-hand side is the square of the Hilbert-Schmidt norm of $\Acal,$ which is $\int_0^1\int_0^1 |K(x,y)-K_h(x,y)|^2dxdy,$ thus proving
                 ~\eqref{eqthmsquaredifference}.

            \end{proof}
            \begin{rem}
            Note that we obtained in particular
            $$\suml_{k=1}^{N-1}|\lambda_{k}^{-1}-\lambda_{h,k}^{-1}|^2\leq Ch^2.$$
               This estimate is valid simultaneously for \textit{all $N-1$} eigenvalues.  Fixing an index $k,$ we get in particular
               \be\label{eqratelambdadiff}
                \frac{ |\lambda_{k}-\lambda_{h,k}|}{\lambda_{h,k}}\leq C \lambda_{k}h.
               \ee
               In view of Claim ~\ref{claimcontevs} we have $\lambda_k \approx k^4.$ Thus ~\eqref{eqratelambdadiff} yields only an $O(h)$ convergence.


      However it is seen in Table ~\ref{table-11}, that even with a small number of grid points, the first discrete eigenvalues approximate very well the continuous ones.
      We shall prove below that indeed the convergence is ``optimal''.

       \begin{table}[]

	\centering
	
	\begin{tabular}{|l|l|l|l|l|}
		\hline
		& k=1        & k=2         & k=3          & k=4          \\ \hline
		\begin{tabular}[c]{@{}l@{}}true\\ eigenvalue\end{tabular} & 500.563902 & 3803.537080 & 14617.630131 & 39943.799006 \\ \hline
		N=10                                                      & 500.521885 & 3800.689969 & 14567.617771 & 39493.816015 \\ \hline
		N=20                                                      & 500.561614 & 3803.398598 & 14615.468848 & 39926.599754 \\ \hline
		N=30                                                      & 500.563462 & 3803.511145 & 14617.236978 & 39940.722654 \\ \hline
		N=40                                                      & 500.563764 & 3803.529031 & 14617.509451 & 39942.881883 \\ \hline
		N=50                                                      & 500.563845 & 3803.533813 & 14617.581402 & 39943.430972 \\ \hline
		N=60                                                      & 500.563874 & 3803.535512 & 14617.606815 & 39943.623511 \\ \hline
	\end{tabular}
	\caption{First $4$ eigenvalues (top row) and their numerical approximations using a grid of $N=10-60$ nodes.}
\label{table-11}
\end{table}

\end{rem}

                We now proceed  to prove the ``optimal'' estimate. Compare ~\eqref{eqestimatefirstev}
          and Remark ~\ref{remestfirstev}
            in what concerns the first eigenvalue.

            \begin{thm}[\textbf{Optimal rate of convergence of discrete eigenvalues}]\label{thmoptimalallevs}
               Fix an integer $k\geq 1$ and consider the discrete eigenvalue $\lambda_{h,k}$  as a function of $h=\frac1N,\,N=k+1,k+2,\ldots.$ Then there exists a constant $C>0,$ depending only on $k,$ such that
                  \be\label{eqrateevh4}
                  |\lambda_k-\lambda_{h,k}|\leq C h^4.
                  \ee
            \end{thm}
            \begin{proof} In view of ~\eqref{eqratelambdadiff} we have
               $$\lim\limits_{h\to 0}\lambda_{h,k}^{-1}=\lambda_k^{-1}.$$
               The sequence $\Lambda^{-1}$ of the reciprocals of the exact eigenvalues (see ~\eqref{eqLmbdamin1}) is monotone decreasing, so there exist $h_0>0$ and $\eta>0,$ such that
                  \be
                    |\lambda_{h,k}^{-1}- \lambda_j^{-1}  |\geq\eta,\quad h<h_0,\,\,j\neq k.
                  \ee
                     Combined with Proposition ~\ref{propdistLambdah} we infer that the only (inverse) eigenvalue that can be ``close'' to $\lambda_k^{-1}$ is $\lambda_{h,k}^{-1},$ and that
                     $$|\lambda_k^{-1}-\lambda_{h,k}^{-1}|\leq Ch^4,$$
                     thus concluding the proof of the theorem.
            \end{proof}

               \begin{rem} Observe that in the proof of Theorem ~\ref{thmoptimalallevs} we relied on special properties of the kernel, via Proposition  ~\ref{propdistLambdah}. Without using such information we obtain ``sub-optimal'' estimates. For example, ~\eqref{eqthmsquaredifference} implies
                     $$\suml_{k=N}^\infty\lambda_{k}^{-2}\leq CN^{-2},$$
                     which is  not optimal, in view of Claim ~\ref{claimcontevs}. Compare also to the estimate in ~\eqref{eqestimgammaGammah} which can be written as
                     $$|\suml_{i=1}^\infty \lambda_i^{-1}-\suml_{i=1}^{N-1} \lambda_{h,i}^{-1}|\leq Ch^4.$$

               \end{rem}
               \begin{rem} The $O(h^4)$ rate of convergence, as stated in Theorem ~\ref{thmoptimalallevs}, can be compared to the method of collocation approximation ~\cite{deboor}. In the case of the latter , achieving a similar rate of convergence requires the construction of an interpolating $C^3$ piecewise fifth-order polynomial function, and then using collocation at Gaussian points. The results here were obtained by using the discretized kernel (of the inverse operator). Owing to the observed connection between this kernel and the classical ($C^2$) cubic splines, the approximating eigenvalues are in fact those of the fourth-order (distributional) derivative of the interpolating cubic spline at the grid points (Proposition ~\ref{cor3rdspljump}).
               \end{rem}
               
               \appendix

                        \section{\textbf{ THE  DISCRETE BIHARMONIC OPERATOR: GENERATING POLYNOMIALS}}\label{sechaggaiexplicit}

          Consider again the discrete fourth-order equation
          \begin{equation}\label{le}\delta^4_x\fu=\ff,\end{equation}
          where $\fu,\,\fu_x\in l^2_{h,0}.$
          In this section we  obtain a direct proof of Corollary ~\ref{cordiscresolvent}.  In other words, we
          compute the matrix corresponding to the operator $(\delta^4_x)^{-1},$ without recourse to the theory of cubic spline functions involved in the previous proof. In fact, an  expression for this matrix has already been given in ~\cite[Section 10.6, Eq. (10.137)]{book} and was used as the main tool in proving Claim ~\ref{claim21}. However, the expression there was a product of three matrices, based on the  matrix representation of the Hermitian derivative. Thus, while allowing to obtain the aforementioned estimates, it did not yield an ``explicit'' form (such that can be used in a computer code in a straightforward way).

           Remarkably, the methodology expounded here uses the discrete operators in a totally different way;
          it employs generating functions, and is a systematic approach that can also be
          applied to other problems. Although the computations involved require some work, it has the advantage of being a straightforward application of the definitions of the discrete operators. It should be mentioned that we first carried out the computation here, and it motivated our search for a parallel ``functional interpreation'', as expressed in
          Corollary ~\ref{cordiscresolvent}.

        By (\ref{d4}), Equation  (\ref{le}) can be rewritten as
\begin{equation}\label{e1}\frac{12}{h^2}\left[\frac{(\fu_{x})_{j+1}-(\fu_{x})_{j-1}}{2h}-\frac{\fu_{j+1}+\fu_{j-1}-2\fu_j}{h^2} \right]=\ff_j,\;\;\;1\leq j\leq N-1,\end{equation}
where by (\ref{eqdefhermit}),
\begin{equation}\label{e2}\frac{1}{6}(\fu_{x})_{j-1}+\frac{2}{3}(\fu_{x})_{j}+\frac{1}{6}(\fu_{x})_{j+1}=\frac{\fu_{j+1}-\fu_{j-1} }{2h},\;\;\;1\leq j\leq N-1,\end{equation}
\begin{equation}\fu_0=\fu_N=(\fu_x)_0=(\fu_x)_N=0.\end{equation}
The system ~\eqref{e1},~\eqref{e2} must be solved for $\set{\fu_j,(\fu_x)_j}_{j=1}^{N-1}.$

To do this, we introduce generating functions, which are polynomials of degree $N-1$ in the variable $z$:
$$p(z)=\sum_{j=1}^{N-1} \fu_jz^j,\;\;\;q(z)=\sum_{j=1}^{N-1} (\fu_x)_j z^j,\;\;\;\phi(z)=\sum_{j=1}^{N-1} \ff_j z^j.$$
We know $\phi(z)$ and want to find $p(z),q(z)$.

Equations (\ref{e1}) can be encoded as the following equality of polynomials,
\be\label{g1}
\frac{1}{2h}\left[(z^{-1}-z)q(z)-(\fu_x)_1+(\fu_x)_{N-1}z^N \right]
-\frac{1}{h^2}\left[(z+z^{-1}-2)p(z)-\fu_1-z^N \fu_{N-1} \right]=\frac{h^2}{12}\phi(z).
\ee
Similarly, Equations (\ref{e2}) are equivalent to the following polynomial equality
\be\label{g2}\left(\frac{1}{6}z^{-1}+\frac{2}{3}+\frac{1}{6}z \right)q(z)-\frac{1}{6}(\fu_x)_1-\frac{1}{6}z^N(\fu_x)_{N-1}
=\frac{1}{2h}\left[(z^{-1}-z)p(z)-\fu_1+z^N\fu_{N-1} \right].\ee
Multiplying (\ref{g1}),(\ref{g2}) by $z$, and rearranging, we have
\be\label{ee1}\aligned
\frac{1}{h^2}(z^2-2z+1)p(z)+\frac{1}{2h}(z^2-1)q(z) \hspace{100pt}\\
=-\frac{h^2}{12}z\phi(z)+\frac{1}{h^2}\left[\fu_1 z+z^{N+1}\fu_{N-1} \right]+\frac{1}{2h}\left[-(\fu_x)_1z+(\fu_x)_{N-1}z^{N+1} \right],
\endaligned\ee
\be\label{ee2}\frac{1}{2h}(z^2-1)p(z) +\left(\frac{1}{6}z^2+\frac{2}{3}z+\frac{1}{6} \right)q(z)
=\frac{1}{2h}\left[-\fu_{1}z+z^{N+1}\fu_{N-1}\right] +\frac{1}{6}(\fu_x)_1z+\frac{1}{6}z^{N+1}(\fu_x)_{N-1}. \ee
We now solve the system of two linear equations (\ref{ee1}),(\ref{ee2}) for $p(z),q(z)$. It suffices to write the solution for $p(z)$, which is
\begin{equation}\label{pe}p(z)=\frac{12h^2}{(z-1)^4}\cdot r(z), \end{equation}
where
\begin{eqnarray}&&\label{rr}r(z)= \left(\frac{1}{6}z^2+\frac{2}{3}z+\frac{1}{6} \right)
\cdot  \frac{h^2}{12}\cdot z\phi(z)\\
&-& \left(\frac{1}{6}z^2+\frac{2}{3}z+\frac{1}{6} \right)\left(\frac{1}{h^2}\left[\fu_{1}z+z^{N+1}\fu_{N-1} \right]+\frac{1}{2h}\left[-(\fu_x)_1z+(\fu_x)_{N-1}z^{N+1} \right]\right)\nonumber\\
&+&\frac{1}{2h}(z^2-1)\left[\frac{1}{2h}\left[-\fu_{1}z+z^{N+1}\fu_{N-1}\right] +\frac{1}{6}(\fu_x)_1z+\frac{1}{6}z^{N+1}(\fu_x)_{N-1} \right]\nonumber\end{eqnarray}

It should be noted that the expression (\ref{rr}) contains the unknown quantities $\fu_{1},\fu_{N-1},(\fu_x)_1,(\fu_x)_{N-1}$. Once we determine these
quantities, (\ref{pe}) will give us the solution to the system (\ref{e1}),(\ref{e2}). To find these quantities, we exploit the
following fact: since $p(z)$ is a polynomial, while the expression (\ref{pe}) contains $(z-1)^4$ in the denominator, it
must be the case that $z=1$ is a root of $r(z)$ of multiplicity $4$, that is
\begin{equation}r(1)=r'(1)=r''(1)=r'''(1)=0.\end{equation}
By differentiating $r(z)$ three times and then substituting $z=1$, we obtain $4$ equations for
$\fu_{1},\fu_{N-1},(\fu_x)_1,(\fu_x)_{N-1}$.
\begin{equation}
\label{r1}r(1)=
\frac{h^2}{12}\cdot \phi(1)- \frac{1}{h^2}\left[\fu_{1}+\fu_{N-1} \right]-\frac{1}{2h}\left[-(\fu_x)_1+(\fu_x)_{N-1} \right]=0
\end{equation}
\begin{eqnarray}\label{r2}&&r'(1)=
h^2\cdot \phi(1)+
 \frac{h^2}{12}\cdot \phi'(1)\\
&-&\frac{1}{h^2}\left[\frac{5}{2}\fu_{1}+\left(N+\frac{3}{2}\right)\fu_{N-1} \right]-\frac{1}{2h}\left[-\frac{7}{3}(\fu_x)_1+\left(N+\frac{5}{3}\right)(\fu_x)_{N-1} \right]=0\nonumber\end{eqnarray}
\begin{eqnarray}\label{r3}&&r''(1)=  \frac{7h^2}{36}\cdot \phi(1)+
\frac{h^2}{3}\cdot \phi'(1)+
 \frac{h^2}{12}\cdot \phi''(1)\\
 &-&\frac{1}{h^2}\left[\frac{23}{6}\fu_1 +\left(\frac{5}{6}+2N+N^2 \right)\fu_{N-1} \right]\nonumber\\
 &-&\frac{1}{2h}\left[ -\frac{10}{3}(\fu_x)_1 +\left( \frac{4}{3}+\frac{7}{3}N+N^2\right)(\fu_x)_{N-1}\right]=0\nonumber\end{eqnarray}
\begin{eqnarray}\label{r4}&&r'''(1)=
   \frac{h^2}{12}\cdot \phi(1)+ \frac{7h^2}{12}\cdot \phi'(1)+
 \frac{h^2}{2}\cdot \phi''(1)+ \frac{h^2}{12}\cdot \phi'''(1)\\
&-&\frac{1}{h^2}\left[ \frac{5}{2}\fu_1  + \left(-\frac{1}{2}+\frac{3}{2}N^2+N^3 \right) \fu_{N-1}\right]\nonumber\\
&-&\frac{1}{2h}\left[ -2(\fu_x)_1  +\left(N+2N^2+N^3 \right) (\fu_x)_{N-1}\right]=0.\nonumber\end{eqnarray}
We set
$$m_0=\phi(1),\;\; m_1=\phi'(1),\;\; m_2=\phi''(1),\;\; m_3=\phi'''(1),$$
solve the linear system (\ref{r1})-(\ref{r4}) for $\fu_{1},\fu_{N-1},(\fu_x)_1,(\fu_x)_{N-1}$, and then
substitute these values into (\ref{rr}),(\ref{pe}), to obtain the expression
\begin{eqnarray}\label{solp}&&p(z)=\frac{1}{6N^4}\Big[\frac{z(z^2+4z+1)}{(z-1)^4}\cdot \phi(z)\\
&+&\frac{z}{(z-1)^4}\Big((-m_0+3m_1)-\frac{1}{N}(6m_1+6m_2)\nonumber\\
&+&\frac{1}{N^2}(6m_1+12m_2+3m_3)-\frac{1}{N^3}(2m_1+6m_2+2m_3) \Big)\nonumber\\
&+&\frac{4z^2}{(z-1)^4}\Big(-m_0+\frac{1}{N^2}(3m_1+3m_2)-\frac{1}{N^3}(2m_1+6m_2+2m_3)\Big)\nonumber\\
&+&\frac{z^3}{(z-1)^4}\Big(-(m_0+3m_1)+\frac{1}{N}(6m_1+6m_2)\nonumber\\&-&\frac{1}{N^2}(6m_2+3m_3)-\frac{1}{N^3}(2m_1+6m_2+2m_3) \Big)\nonumber\\
&+&\frac{z^{N+1}}{(z-1)^4}\Big(-\frac{1}{N} (3m_1+3m_2) +\frac{1}{N^2}(6m_2+3m_3)  +\frac{1}{N^3}(2m_1+6m_2 +2m_3) \Big)\nonumber\\
&+&\frac{z^{N+2}}{(z-1)^4}\Big(\frac{1}{N^2}(-12m_1-12m_2) +\frac{1}{N^3}(8m_1+24m_2+8m_3) \Big)\nonumber\\
&+&\frac{z^{N+3}}{(z-1)^4}\Big(\frac{1}{N}(3m_1+3m_2)-\frac{1}{N^2}(6m_1+12m_2+3m_3)\nonumber\\&+&\frac{1}{N^3}(2m_1+6m_2+2m_3) \Big)\Big].\nonumber\end{eqnarray}
Note that since $p(z)$ is a polynomial of degree $N-1$, all terms $z^j$ ($j\geq N$) in fact cancel.
We explicitly compute the coefficient of the term $z^j$ ($1\leq j\leq N-1$) in $p(z)$, which gives us $\fu_j$.
Using
$$\frac{1}{(z-1)^4}= \frac{1}{6}\sum_{j=0}^\infty (j+1)(j+2)(j+3) z^{j},\;\;\;\frac{z}{(z-1)^4}= \frac{1}{6}\sum_{j=0}^\infty j(j+1)(j+2) z^{j},$$
$$\frac{z^2}{(z-1)^4}= \frac{1}{6}\sum_{j=0}^\infty (j-1)j(j+1) z^{j},\;\;\;\frac{z^3}{(z-1)^4}= \frac{1}{6}\sum_{j=0}^\infty (j-2)(j-1)j z^{j},$$
$$\frac{z(z^2+4z+1)}{(z-1)^4}  =\sum_{j=0}^\infty j^3 z^{j},  $$
we have
$$\frac{z(z^2+4z+1)}{(z-1)^4}\phi(z)=\sum_{j=1}^{N-1} z^j \sum_{l=1}^{j-1} (j-l)^3 \ff_l+ ({\mbox{terms of order }}\geq N),$$
so that the coefficient of $z^j$ ($1\leq j\leq N-1$) in $p(z)$ is
$$\fu_j=\frac{j}{6N^4}\Big[  \frac{1}{j}\sum_{l=1}^{j-1} (j-l)^3 \ff_l$$
$$+\frac{1}{6}(j+1)(j+2)\Big((-m_0+3m_1)-\frac{1}{N}(6m_1+6m_2)$$
$$+\frac{1}{N^2}(6m_1+12m_2+3m_3)-\frac{1}{N^3}(2m_1+6m_2+2m_3) \Big)$$
$$+\frac{2}{3}(j-1)(j+1)\Big(-m_0+\frac{1}{N^2}(3m_1+3m_2)-\frac{1}{N^3}(2m_1+6m_2+2m_3)\Big)$$
$$+\frac{1}{6}(j-2)(j-1)\Big(-(m_0+3m_1)+\frac{1}{N}(6m_1+6m_2)-\frac{1}{N^2}(6m_2+3m_3)-\frac{1}{N^3}(2m_1+6m_2+2m_3) \Big)\Big]$$
$$=\frac{j}{6N^4}\Big[  \frac{1}{j}\sum_{l=1}^{j-1} (j-l)^3 \ff_l+j\left(-j m_0+3 m_1\right)-\frac{1}{N}\cdot 6j(m_1+m_2)$$
$$+\frac{3j}{N^2}\left((j+1)m_1+(j+3)m_2+m_3 \right)-\frac{2j^2}{N^3}\left( m_1+3m_2+m_3 \right)\Big].$$
We now note that
$$m_0=\phi(1)=\sum_{k=1}^{N-1} \ff_k,\;\;\;m_1=\phi'(1)=\sum_{k=1}^{N-1} k\ff_k,$$
$$m_2=\phi''(1)=\sum_{k=2}^{N-1} k(k-1)\ff_k,\;\;\;m_3=\phi'''(1)=\sum_{k=2}^{N-1} k(k-1)(k-2)\ff_k,$$
so that
$$m_1+m_2= \sum_{k=1}^{N-1} [k+k(k-1)]\ff_k=\sum_{k=1}^{N-1} k^2\ff_k $$
$$m_1+3m_2+m_3=\sum_{k=1}^{N-1} k^3\ff_k,$$
and using these results the expression for $\fu_j$ simplifies to
$$\fu_j=\frac{1}{6N}\Big[  \sum_{k=1}^{j-1} \left( \frac{j}{N}-\frac{k}{N}\right)^3 \ff_k- \left(\frac{j}{N}\right)^2\cdot \sum_{k=1}^{N-1} \left(1-\frac{k}{N} \right)\left(2\frac{k}{N}\cdot \frac{j}{N}+ \frac{j}{N}-3\frac{k}{N} \right)\ff_k\Big]$$
$$=\frac{1}{6N}\Big[  \sum_{k=1}^{j-1} \left( x_j-x_k\right)^3 \ff_k+ x_j^2\cdot \sum_{k=1}^{N-1} \left(1-x_k \right)^2\left(2(1-x_j)x_k+x_k-x_j \right)\ff_k\Big]$$
$$=\frac{1}{6N}\Big[ (1-x_j)^2 \cdot \sum_{k=1}^{j-1} x_k^2\left(2x_j(1-x_k)+x_j-x_k \right) \ff_k+ x_j^2\cdot \sum_{k=j}^{N-1} \left(1-x_k \right)^2\left(2x_k(1-x_j)+x_k-x_j \right)\ff_k\Big].$$
We have thus obtained
\begin{prop}\label{formula} Defining the matrix elements
$$K^h_{j,k}=\begin{cases}\frac{1}{6N}\cdot (1-x_j)^2\cdot x_k^2 \left(2x_j(1-x_k)+x_j-x_k \right),& 1\leq k \leq j\leq N-1,\\
\frac{1}{6N}\cdot (1-x_k)^2\cdot x_j^2 \left(2x_k(1-x_j)+x_k-x_j \right),&1\leq j \leq k\leq N-1,
\end{cases}$$
we have that the solution of (\ref{le}) is given by
$$\fu_j= \sum_{k=1}^{N-1} K^h_{j,k} \ff_k.$$
\end{prop}

This expression is seen to be identical to ~\eqref{eqschwartzkernel},
so that Proposition ~\ref{formula} is a re-statement of Corollary ~\ref{cordiscresolvent}.



%


\begin{thebibliography}{100}
  \bibitem{spline-book}J. H. Ahlberg, E. N. Nilson and J. L. Walsh, {\sl ``The Theory of Splines and Their Applications'', }
 Academic-Press ,1967.

 \bibitem{andrew}A. L. Andrew and J. W. Paine, Correction of finite element estimates for  Sturm-Liouville eigenvalues,
  Numer. Math. {\bf 50} (1986), 205--215.
\bibitem{book}M. Ben-Artzi, J.-P. Croisille and D. Fishelov,
    {\sl ``Navier-Stokes Equations in Planar Domains'',} Imperial College Press, 2013.
\bibitem{sturm}M. Ben-Artzi, J.-P. Croisille , D. Fishelov and R. Katzir,
      Discrete fourth-order Sturm-Liouville problems, IMA J. Numer.Anal. (to appear 2017).


     \bibitem{boumenir}A. Boumenir, Sampling for the fourth-order Sturm-Liouville differential operator,
   J. Math. Anal. Appl.
    {\bf 278} (2003), 542--550.
    \bibitem{davies}E. B. Davies,
    {\sl ``Spectral Theory and  Differential Operators'',} Cambridge University Press, 1995.
    \bibitem{deboor0}C. de Boor,
    {\sl ``A Practical Guide to Splines-Revised Edition'',} Springer New York, 2001.
\bibitem{deboor}C. de Boor and B. Swartz,  Collocation approximation to eigenvalues of an ordinary differential equation: The principle of the thing,
    Math. Comp.
    {\bf 35} (1980), 679--694.

\bibitem{evans}L.C. Evans, {\sl "Partial Differential Equations", }
 American Mathematical Society ,1998.

 \bibitem{everitt}W. N. Everitt, The Sturm-Liouville problem for fourth-order differential equations,
   The Quar. J. Math.
    {\bf 8} (1957), 146--160.

 \bibitem{abarbanel}D. Fishelov, M. Ben-Artzi and J.-P. Croisille  , Recent advances in the study of a fourth-order compact scheme for the one-dimensional biharmonic equation,
  J. Sci. Comput.  {\bf 53} (2012), 55--79.

  \bibitem{Grunau}H.-C. Grunau and F. Robert  , Positivity and almost positivity  of biharmonic Green's functions under Dirichlet boundary conditions,
 Arch. Rat. Mech. Anal.  {\bf 196} (2010), 865--898.


 \bibitem{kato-book}T. Kato, {\sl ``Perturbation Theory for Linear Operators'', }
 Springer-Verlag ,1980.

  \bibitem{kato}T. Kato,  Variation of discrete spectra,
   Commun. Math. Phys
    {\bf 111} (1987), 501--504.
 \bibitem{markus}A. S. Markus,  The eigen- and singular values of the sum and product of linear operators,
   Russian Mathematical Surveys
    {\bf 19} (1964), 91--120.
 \bibitem{pipher}J. Pipher and G. Verchota,  A maximum principle for biharmonic functions in Lipschitz and $C^1$ domains,
   Comment. Math. Helvetici
    {\bf 68} (1993), 384--414.
     \bibitem{rattana}A. Rattana and C. B\"{o}ckmann, Matrix methods for computing eigenvalues of Sturm-Liouville problems of order four,
  J. Comp. and Applied Math. {\bf 249} (2013), 144--156.
  \bibitem{nystrom}A. Spence,  On the convergence of the Nystr\"{o}m method for the integral equation eigenvalue problem,
   Numer. Math.
    {\bf 25} (1975), 57--66.

 \end{thebibliography}
\end{document}